\documentclass{article}    %
\usepackage{arxiv}
\usepackage[utf8]{inputenc}
\usepackage{microtype}
\usepackage{graphicx,caption}
\usepackage{booktabs,multirow,adjustbox} %
\usepackage{amsmath,amssymb,amsthm}
\usepackage{url}
\usepackage{tikz}
\usepackage{enumitem}
\usepackage{mathrsfs}
\usepackage{subfig}
\usepackage{wrapfig}
\usepackage[algo2e,ruled, vlined, linesnumbered]{algorithm2e} %
\setlist{nosep} %

\newtheorem{mydef}{Definition}
\newtheorem{lem}{Lemma}

\newtheorem{prop}{Proposition}
\newtheorem{prob}{Problem}

\newtheorem{thm}{Theorem}
\newtheorem{assum}{Assumption}

\usepackage{todonotes}

\newcommand{\Nint}[1]{[ #1]}
\newcommand{\diam}{\mathrm{diam}}

\newcommand{\argmin}{\operatorname{arg\ min}}

\newcommand{\arginf}{\operatorname{arg\ inf}}

\newcommand{\Npos}{\mathbb{N}_{+}}
\newcommand{\ftsc}{f_{t,\mu}^-} %
\newcommand{\Htsc}{H_{t,\mu}^-} %

\newcommand{\xetaplusminusdelta}{x_{\eta^-,\eta^+,\delta}}
\newcommand{\xopt}{x^\ast}

\newcommand{\xoptdelta}{x_{\delta}^\ast}

\newcommand{\xddagglobal}{x^\ddag_{\mathrm{global}}}
\newcommand{\Deltaglobal}{\Delta_{\mathrm{global}}}

\newcommand{\Gfmax}{{\|\nabla f\|}_{\mathrm{max}}}
\newcommand{\GJmax}{{\|\nabla J\|}_{\mathrm{max}}}
\newcommand{\GHmax}{{\|\nabla H\|}_{\mathrm{max}}}

\newcommand{\LH}{L_H}

\newcommand{\FH}{\mathscr{F}_{L_H}}
\newcommand{\FHmu}{\mathscr{F}_{\LH}^{\mu}}

\newcommand{\FJ}{\mathscr{F}_{L_J}}

\newcommand{\FK}{\mathscr{F}_{K}}
\newcommand{\FKJ}{\mathscr{F}_{K_J}}
\newcommand{\Ff}{\mathscr{F}_{L_f}}
\newcommand{\Ffmu}{\mathscr{F}_{L_f}^\mu}

\newcommand{\Tsuff}{T_\mathrm{sufficient}}
\newcommand{\Tsuffcov}{T}
\newcommand{\Tsuffsc}{T_{\mathrm{sufficient},\mu\text{-}\mathrm{convex}}}
\newcommand{\datapointsi}{J(q_k), H(q_k),
\nabla J(q_k), \nabla H(q_k)}

\newcommand{\mycovmethodqi}{{\{q_i: i\in\Nint{t+1},
H(q_i)\leq \delta\}}}

\usepackage{algorithmic,algorithm} %

\makeatletter
\define@key{Gin}{Trim}{\let\Gin@viewport@code\Gin@trim\expandafter\Gread@parse@vp#1 \\}
\makeatother

\title{Constrained, Global Optimization of Functions
with Lipschitz Continuous Gradients}
\author{Abraham P. Vinod, Arie Israel, and Ufuk
Topcu\thanks{A. Vinod is with Mitsubishi Electric Research
    Laboratories (MERL), Cambridge, MA, 02139, USA. A.
    Israel is with the Department of Mathematics, and U.
    Topcu is with the Department of Aerospace Engineering
    and Engineering Mechanics at the University of Texas at
    Austin, Austin, TX, 78712 USA; email:
    aby.vinod@gmail.com, arie@math.utexas.edu,
    utopcu@utexas.edu.\newline This work was completed while
    Vinod was a postdoctoral research fellow at the
    University of Texas at Austin.}}
\date{}
\renewcommand{\headeright}{}
\renewcommand{\undertitle}{}

\begin{document}

\maketitle
\begin{abstract}
    We present two first-order, sequential optimization algorithms to solve constrained optimization problems. We consider a black-box setting with a priori unknown, non-convex objective and constraint functions that have Lipschitz continuous gradients. The proposed algorithms balance the exploration of the a priori unknown feasible space with the pursuit of global optimality within in a pre-specified finite number of first-order oracle calls. The first algorithm accommodates an infeasible start, and provides either a near-optimal global solution or establishes infeasibility. However, the algorithm may produce infeasible iterates during the search. For a strongly-convex constraint function and a feasible initial solution guess, the second algorithm returns a near-optimal global solution without any constraint violation. In contrast to existing methods, both of the algorithms also compute global suboptimality bounds at every iteration. They can satisfy user-specified tolerances in the computed solution with near-optimal complexity in oracle calls for a large class of optimization problems. We propose tractable implementations of the algorithms by exploiting the structure afforded by the Lipschitz continuous gradient property.

\end{abstract}
\section{Introduction}

We study first-order methods to solve the following
constrained, global optimization problem,
\begin{align}
    \mathrm{minimize}\ J(x)\quad\mathrm{subject\
        to}\ H(x)\leq 0\label{prob:orig_prob},
\end{align}
where the functions $J,H: \mathbb{R}^d \to \mathbb{R}$ are
(possibly non-convex) functions with Lipschitz continuous
gradients. We consider the black-box setting, where the 
functions $J$ and $H$ are \emph{a priori} unknown, and are 
accessible only via first-order oracles. 
We denote a global minimum of \eqref{prob:orig_prob} by
$\xopt$.  We propose two sequential optimization algorithms
that approximate $\xopt$ in a finite number of oracle
queries. Unlike existing methods, the algorithms provide
global suboptimality bounds at every iteration which enable
early termination, and the algorithms are worst-case optimal
in the budget of the oracle  calls required to achieve
user-specified tolerances for a large class of problems.

Constrained, global optimization problems of the form
\eqref{prob:orig_prob} are ubiquitous in science and
engineering. An application of \eqref{prob:orig_prob} in
machine learning arises in policy optimization for
reinforcement learning. Here, we maximize a long-term reward
associated with the learning problem by optimizing a
parameterized policy, typically a neural
network~\cite{sutton2018reinforcement}. The constraint $H$ 
in such problems can impose additional desirable properties 
or domain-specific knowledge on the policy network. As an
illustration, we train a neural network to solve the
mountain car problem~\cite{openaigym} in
Section~\ref{sub:mountain}, and demonstrate that
imposing minimum energy requirements on the closed-loop
system allows completion of the task with very few
simulations.

A significant part of existing research on global
optimization focuses on a special case of
\eqref{prob:orig_prob}~\cite{horst2000introduction,pardalos2010deterministic,kochenderfer2019algorithms},
\begin{align}
\mathrm{minimize}\ J(x)\quad \mathrm{subject\ to}\ x \in
\mathcal{X}\label{prob:orig_prob_noH}.
\end{align}
where $ \mathcal{X}\subset
\mathbb{R}^d$ is a known, convex, and compact set.
A popular approach to tackle \eqref{prob:orig_prob_noH} in a
black-box setting is via iterative optimization of a surrogate optimization
problem, constructed using the regularity of $J$ and the
information from the past oracle queries. For example, see Piyavskii-Shubert
algorithm~\cite{piyavskii_algorithm_1972,shubert1972sequential}
and \texttt{DIRECT}~\cite{jones1993lipschitzian} for
\emph{Lipschitz continuous} $J$, and covering
methods~\cite{blanquero_covering_2000,
fortin_piecewise_2011,pardalos2010deterministic} for $J$ with \emph{Lipschitz
continuous gradients}. %
Alternatively, researchers have utilized hierarchical
partitioning of $ \mathcal{X}$ to design optimistic
optimization algorithms~\cite{munos2011optimistic} that solve
\eqref{prob:orig_prob_noH}. The Piyavskii-Shubert algorithm,
the covering methods, and the optimistic optimization
algorithms have deterministic bounds on the global suboptimality 
for a given budget of oracle 
calls~\cite{ivanov1972optimal,munos2011optimistic}.
On the other hand, random search algorithms utilize Lipschitz
information to provide probabilistic budget-dependent bounds
on suboptimality~\cite{malherbe2017global}. In this paper,
we develop novel techniques to perform constrained, global
optimization of \eqref{prob:orig_prob} in a
black-box setting, inspired by the covering methods.

Bayesian optimization is another popular sequential optimization approach for black-box
optimization~\cite{eriksson2020scalable,gardner2014bayesian,gelbart2014bayesian,mockus2012bayesian,snoek2012practical,Wu2017GPgrad}.
It implicitly imposes regularity
requirements on the \emph{a priori} unknown objective and
constraint functions by modeling them as samples drawn from
a fixed Gaussian processes~\cite{mockus2012bayesian,rasmussen2006gaussian}.
At every iteration, it constructs an \emph{acquisition
function} using the data from past queries, and solves a
surrogate optimization problem to identify the next query
point. For the optimization problem
\eqref{prob:orig_prob_noH}, existing literature provides
budget-dependent, probabilistic-suboptimality bounds on the
estimated optimum via \emph{regret
bounds}~\cite{FreitasExpoential2012,srinivas2012information}.
For the constrained optimization problem
\eqref{prob:orig_prob}, the acquisition function is
multiplied with another surrogate function, which models the
\emph{probability of feasibility}~\cite{eriksson2020scalable,gardner2014bayesian,gelbart2014bayesian}.
To the best of our knowledge, such approaches 
do not have any convergence guarantees or budget-dependent global 
suboptimality bounds guarantees. The main advantages of the
proposed algorithms presented here over Bayesian
optimization techniques are as follows: 1) global suboptimality bounds available from the first feasible iteration, 2) sufficient budgets for the algorithm to
achieve user-specified solution tolerances or demonstrate near-infeasibility, and
3) constraint-violation-free optimization of
\eqref{prob:orig_prob}, when $H$ is additionally known to be strongly-convex.
Similar to Bayesian optimization problem, the proposed
algorithms also solve
non-convex surrogate optimization problems. However, due to
the structure afforded by Lipschitz gradient continuity, the
resulting problems are simpler non-convex, quadratically constrained,
quadratic programs, that can be efficiently handled using
existing off-the-shelf solvers, like \texttt{GUROBI}.

The main contributions of this paper are \emph{two first-order, sequential optimization
algorithms that approximate the global minimum of \eqref{prob:orig_prob} with valid global
suboptimality bounds under a finite budget of oracle calls.} Starting with a (possibly
infeasible) initial solution guess, the first algorithm
approximates $\xopt$ or proves the (near-)infeasibility of
\eqref{prob:orig_prob}. The first algorithm does not require
the initial solution guess to be feasible for
\eqref{prob:orig_prob}. In contrast, the second algorithm
solves \eqref{prob:orig_prob} without any constraint
violation, when the constraint function $H$ is
strongly-convex and the initial solution guess is feasible
for \eqref{prob:orig_prob}. Both of the algorithms are
\emph{anytime}, i.e., they can be terminated at any point of
time to return a valid approximation of $\xopt$ with global
suboptimality bound, or a near-infeasibility certificate. We
also characterize worst-case, sufficient budgets of oracle
calls for the algorithms to achieve a user-specified,
global-suboptimality bounds, and show that they are tight up
to a constant factor for a large class of problems.

The rest of this paper is organized as follows.
Section~\ref{sec:prelim} states the problems of interest,
and provides a brief description of mathematical concepts
and existing work relevant to solve \eqref{prob:orig_prob}.
Section~\ref{sec:main} provides the main results of this
paper --- two algorithms to solve \eqref{prob:orig_prob}
along with the proofs of correctness and a discussion about
their implementation. We investigate the efficacy of the
proposed algorithms in numerical experiments in
Section~\ref{sec:num}, and conclude in
Section~\ref{sec:conc}.

\section{Setup and preliminaries}
\label{sec:prelim}

We denote the set of natural and real numbers by
$\mathbb{N}$ and $ \mathbb{R}$ respectively, the
set of natural numbers (not including zero) by $\Npos$, and
the set of non-zero natural numbers up to $t\in\Npos$ by
$\Nint{t}=\{1,2,\ldots, t\}$. For any set
$ \mathcal{S}$, $ \mathcal{S}^d$ refers to the Cartesian
product of $ \mathcal{S}$ with itself $d$-times. We denote
the cardinality of a finite set $ \mathcal{S}$ by $|
\mathcal{S}|$, and the absolute value of a scalar $x \in
\mathbb{R}$ by $|x|$. We use $\|x\|$ to denote the Euclidean 
norm of a vector $x\in\mathbb{R}^d$, and denote the inner between two vectors $x,y$ by $x\cdot y$. Given a compact set 
$\mathcal{X}\subset \mathbb{R}^d$, we define its diameter as
$ \diam( \mathcal{X}) =\sup_{y,x\in \mathcal{X}} \|y -x\|$.
The first-order approximation $\ell:
\mathbb{R}^d \to \mathbb{R}$ of a
continuously-differentiable function $f: \mathcal{X} \to
\mathbb{R}$ about a point $q\in \mathcal{X}$ is given by
\begin{align}
    \ell(x;q,f)&\triangleq f(q) + \nabla f(q)\cdot(x -
    q)\label{eq:first_order}.
\end{align}
Let $\Gfmax$ denote the finite upper bound on $\|\nabla f\|$
over a compact $ \mathcal{X}$,
\begin{align}
    \Gfmax\triangleq\sup_{x\in \mathcal{X}}
    \|\nabla f(x)\|<\infty.\label{eq:GXmax}
\end{align}
We also recall that for any
differentiable $f: \mathcal{X} \to \mathbb{R}$ and any point
$x,y\in \mathcal{X}$, there exists a point $z$ on the line
joining $x$ and $y$, such that
\begin{align}
    |f(y) - f(x)|&= \nabla f(z) \cdot (y - x) \leq
    \|\nabla f(z)\|\|y-x\|\leq
    \Gfmax\|y-x\|\label{eq:mvt_cs}.
\end{align}
Equation \eqref{eq:mvt_cs} follows from mean
value theorem and Cauchy-Schwartz inequality.

\emph{Lipschitz continuous
gradient}~\cite{nesterov_lectures_2018}: Given a set $
\mathcal{X}\subset \mathbb{R}^d$, a 
continuously-differentiable function $f: \mathcal{X} \to
\mathbb{R}$ has a \emph{Lipschitz continuous gradient}, if
its gradient $\nabla f$ satisfies the property $\|\nabla
f(y) - \nabla f(x)\|\leq K_f\|y - x\|$ for every $x,y \in
\mathcal{X}$ for the smallest constant $K_f \in \mathbb{R},\
K_f\geq 0$. We define a \emph{Lipschitz gradient constant}
$L_f$ as any known upper bound on $K_f$, 
since $K_f$ is rarely known. We denote the family of
functions $f$ with Lipschitz gradient constant $L_f$ by $
\mathscr{F}_{L_f}$. 
For brevity, we will refer to functions with Lipschitz
continuous gradients as \emph{smooth functions}.

\emph{Strong-convexity
($\mu$-convexity)}~\cite{nesterov_lectures_2018}: Given a
set $ \mathcal{X}\subset \mathbb{R}^d$, a 
continuously-differentiable function $f: \mathcal{X} \to
\mathbb{R}$ is \emph{strongly-convex} or
\emph{$\mu$-convex}, if for any $x,y\in \mathcal{X}$,
\begin{align}
    f(y)&\geq f(x) + \nabla f(x) \cdot (y-x) + \frac{\mu}{2}
    \|y - x\|^2\label{eq:strong_cvx},
\end{align}
for some \emph{convexity
constant} $\mu>0$. Similarly to the Lipschitz gradient
constant, we do not require $\mu$ to be the largest positive
scalar satisfying \eqref{eq:strong_cvx} for every $x,y\in
\mathcal{X}$.  When $f$ is also smooth with Lipschitz
gradient constant $L_f$, then $\mu \leq L_f$.  We use
$\Ffmu\subset\Ff$ to denote the family of $\mu$-convex, 
$L_f$-smooth functions.

\subsection{Sequential optimization algorithms}
\label{sub:seqoptalg}

Sequential optimization algorithms are popular due
to their ease in design, implementation, and analysis. In this paper,
we study \emph{first-order, sequential optimization algorithms}
to solve \eqref{prob:orig_prob}.
\begin{mydef}{\textsc{(First-order, sequential optimization algorithm)}}
    Given an initial solution guess $q_1\in \mathcal{X}$ and
    a budget $T\in \Npos$ of oracle calls, a
    \emph{first-order, sequential optimization algorithm} is
    a procedure that generates a sequence of query points
    $\{q_t\}_{t=2}^T$ for the first-order oracles for $J$
    and $H$. At every
    iteration $t\in[T-1]$, the
    algorithm constructs $q_{t+1}$ using the information
    available until then 
    ${\{\datapointsi\}}$ with ${k\in\Nint{t}}$. The algorithm computes
    $\xopt$ (or an approximation) within $T$ iterations.
\end{mydef}
Examples of first-order, sequential optimization algorithms
include gradient descent and sequential quadratic
programming~\cite{pmlr-v28-jaggi13,nesterov_lectures_2018,nocedal_numerical_2006}.

Unfortunately, due to the richness of the family of
smooth functions, even the
computation of a feasible solution to
\eqref{prob:orig_prob}
using \emph{any} first-order, sequential optimization algorithm can be
arbitrarily difficult under a fixed budget of oracle calls.
See Appendix~\ref{app:adversarial_example} for such an
``adversarial'' example. Therefore, we will focus on the
computation of an $(\eta,\eta,\delta)$-minimum of
\eqref{prob:orig_prob} or proving
$\gamma$-infeasibility (near-infeasibility for small
$\gamma$).
\begin{mydef}[\textsc{$(\eta^-,\eta^+,\delta)$-minimum and $(\eta,\delta)$-minimum of
\eqref{prob:orig_prob}}]
    Given $\eta^-,\eta^+,\delta\geq 0$, a solution
    $\xetaplusminusdelta\in \mathcal{X}$ is an
    $(\eta^-,\eta^+,\delta)$-minimum  of
    \eqref{prob:orig_prob}, provided
    \begin{align}
        -\eta^- \leq J(\xetaplusminusdelta) - J(\xopt)
        &\leq \eta^+\ \mbox{ and }\
        H(\xetaplusminusdelta)\leq
        \delta\label{eq:desired_conditions},
    \end{align}
    The solution $\xetaplusminusdelta$ is the $(\eta,\delta)$-minimum of
    \eqref{prob:orig_prob}, when $\eta^-=\eta^+=\eta$.
\end{mydef}
\begin{mydef}[\textsc{$\gamma$-infeasibility of
    \eqref{prob:orig_prob}}]\label{defn:gamma_infeas}
    We declare \eqref{prob:orig_prob} to be
    $\gamma$-infeasible for some $\gamma\geq0$, when the
    following optimization is infeasible,
    \begin{align}
        \mathrm{minimize}\ J(x)\quad\mathrm{subject\
            to}\ H(x)<
            -\gamma\label{prob:orig_prob_multiple_infeasibility}.
    \end{align}
\end{mydef}

\subsection{Problem statements}

To ensure that \eqref{prob:orig_prob} does not
have an unbounded solution, we make the following standing
assumption throughout the paper.

\begin{assum}[\textsc{Feasible space of
    \eqref{prob:orig_prob} lies inside a known, convex and
compact set}]\label{assum:X}
    We assume the knowledge of a convex and compact set $
    \mathcal{X}\neq \emptyset$ that contains the \emph{a
    priori} unknown feasible set $\{H\leq 0\}$ of
    \eqref{prob:orig_prob}.
\end{assum}
When the constraint set $\{H\leq 0\}$ is unbounded, we will
seek the (local) minimum of \eqref{prob:orig_prob} inside
the set $\mathcal{X}\cap \{H\leq 0\}$.

Apart from the knowledge of $
\mathcal{X}$, we will assume access to the following to
solve \eqref{prob:orig_prob}:
        1) the first-order oracles for the \emph{a
        priori} unknown functions $J$ and $H$
        that provide $(J(q),\nabla J(q))$ and ${(H(q),\nabla
        H(q))}$ at any query point $q\in \mathcal{X}$
        respectively,
        2) Lipschitz gradient constants $L_J$ and $L_H$
        for $J$ and $H$ over $ \mathcal{X}$ respectively, and
        3) an initial solution guess $q_1\in \mathcal{X}$
        that may be infeasible for \eqref{prob:orig_prob}.
We now state the two problems of interest.

\renewcommand{\theprob}{\Alph{prob}}
\begin{prob}[\textsc{Global optimization for smooth $H$}]\label{p_st:algo}
    Given a budget $T\in\Npos$ of oracle calls and a
    relaxation threshold $\delta > 0$, design a
    first-order, sequential optimization algorithm that
    either declares \eqref{prob:orig_prob} to be
    $\gamma$-infeasible for some $\gamma > 0$, or computes a
    $(\Deltaglobal,\delta)$-minimum of
    \eqref{prob:orig_prob} for some
    $\Deltaglobal>0$. Also, 
    given a
    global-suboptimality threshold $\eta > 0$,
    characterize
    the budget $\Tsuff$ of oracle calls needed by the
    algorithm to compute an
    $(\eta,\delta)$-minimum (when it exists) or declare
    \eqref{prob:orig_prob} to be infeasible,
    irrespective of the choice of $J$, $H$, and $q_1$.
\end{prob}
We also consider a special case of Problem~\ref{p_st:algo}, where 
the \emph{a priori} unknown constraint function $H$ is strongly-convex, and 
the initial solution guess $q_1\in \mathcal{X}$ is feasible for \eqref{prob:orig_prob}.
Here, we assume that $H\in\FHmu$ with known constants $L_H\geq \mu > 0$.
Problem~\ref{p_st:algo_safe} 
searches for the global minimum of \eqref{prob:orig_prob}, without
violating the \emph{a priori} unknown constraint $H\leq
0$ in \eqref{prob:orig_prob}. 
\begin{prob}[\textsc{Global optimization for smooth,
    strongly-convex $H$}]\label{p_st:algo_safe} Given a budget
    $T\in\Npos$ of oracle calls, $H\in\FHmu$, and a feasible 
    initial solution guess, design a
    first-order, sequential optimization algorithm that
    computes a $(0,\Deltaglobal,0)$-minimum of
    \eqref{prob:orig_prob} for some
    $\Deltaglobal>0$ without violating the constraint
    $H\leq 0$ at any iteration. Also, characterize the
    budget $\Tsuffsc$ of oracle calls needed by the
    algorithm to compute an $(0,\eta,0)$-minimum,
    irrespective of the choice of $J$, $H$, and $q_1$.
\end{prob}

\subsection{Data-driven approximants for smooth functions} 
\label{sub:minorant_majorant}

\begin{wrapfigure}{r}{.45\textwidth}
    \definecolor{pydarkgreen}{RGB}{0, 100, 0}
    \definecolor{pydarkorange}{RGB}{255, 140, 0}
    \centering
    \includegraphics[width=0.99\linewidth]{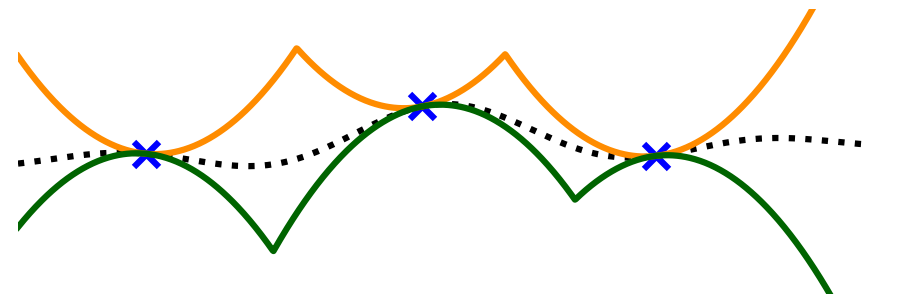}
    \begin{tikzpicture}[overlay,remember picture]
        \small
        \node[anchor=west, color=pydarkorange] (fp) at
            (2.6,2.4) {$\boldsymbol{f_3^+}$};
        \node[anchor=west] (f) at (2.6,1.7)
            {$\boldsymbol{f}$};
        \node[anchor=west, color=pydarkgreen] (fn) at
            (2.6,0.5) {$\boldsymbol{f_3^-}$};
        \node[anchor=west, color=blue] (q1) at (-2.3,1.7)
            {$q_1$};
        \node[anchor=west, color=blue] (q2) at (-0.4,2)
            {$q_2$};
        \node[anchor=west, color=blue] (q3) at (1.2,1.65)
            {$q_3$};
    \end{tikzpicture}
    \caption{Approximants for $f\in \Ff$}\label{fig:bounds}
    \vspace*{-1.5em}
\end{wrapfigure}
 
For any function $f:
\mathcal{X}\to \mathbb{R}$, we define its \emph{minorants}
and \emph{majorants} as functions ${f^-,f^+: \mathcal{X} \to
\mathbb{R}}$ respectively,
\begin{align}
    f^-(x)&\leq f(x)\leq f^+(x),&\forall x\in \mathcal{X}.
\end{align}
Lemma~\ref{lem:approx} constructs majorants and minorants of
a function $f\in\Ff$ using data as shown in
Figure~\ref{fig:bounds}.

\begin{lem}[\textsc{Majorant and minorant for $f$}]
\label{lem:approx} 
    Consider a function $f: \mathcal{X} \to \mathbb{R}$,
    where $f\in \Ff$. Given $t \in \Npos$ and data $\{(q_i,
    f(q_i), \nabla f(q_i)\}_{i=1}^{t}$, a majorant and
    minorant of $f$ is given by,
\begin{align}
    f^+_t(x) &=\min\limits_{i\in \Nint{t}}\left(\ell(x;q_i,f)+\frac{L_f}{2} \| x -
    q_i\|^2\right),\ \mbox{ and }\ \label{eq:major_f}\\
        f^-_t(x) &=\max\limits_{i\in \Nint{t}}\left(\ell(x;q_i,f)-\frac{L_f}{2} \| x -
   q_i\|^2\right), \label{eq:minor_f}%
\end{align}
    respectively. Furthermore, $f^+_t(q_i)=f^-_t(q_i)=f(q)$
    for every $i\in \Nint{t}$, and the approximation errors
    $f^+_t(x)-f(x)$ and $f(x)-f^-_t(x)$ lie in a bounded
    interval $[0, \min\limits_{i\in \Nint{t}} L_f \| x -
    q_i\|^2]$.
\end{lem}
\begin{proof}
    For any smooth function $f$, the following inequalities
    hold for any $x \in \mathcal{X}$,
    \begin{subequations}
        \begin{align}
            f(x) &\leq \ell(x;q_i,f) +
            \frac{L_f}{2}\|x - q_i\|^2,\qquad \forall i\in\Nint{t}
            \label{eq:approx_major_interim},\\
            f(x) &\geq \ell(x;q_i,f)
            - \frac{L_f}{2}\|x - q_i\|^2,\qquad \forall i\in\Nint{t}
            \label{eq:approx_minor_interim}.
        \end{align}%
    \end{subequations}%
    We obtain the data-driven majorant $f^+$
    \eqref{eq:major_f} and minorant $f^-$ \eqref{eq:minor_f}
    via finite minimum of \eqref{eq:approx_major_interim}
    and finite maximum of \eqref{eq:approx_minor_interim}
    over $i\in\Nint{t}$ respectively. By construction, these
    piecewise-quadratic functions coincide with $f(q_i)$ at
    $x=q_i$ for every $i\in\Nint{t}$. Also,
    \begin{align}
        \eqref{eq:approx_major_interim} &\Rightarrow  f(x) - \left({\ell(x;q_i, f) -
    \frac{L_f}{2}\|x - q_i\|^2}\right)\leq L_f\|x -
        q_i\|^2\label{eq:approx_minor_interim2},\\
        \eqref{eq:approx_minor_interim} &\Rightarrow
        \left({\ell(x;q_i, f) +
        \frac{L_f}{2}\|x - q_i\|^2}\right) - f(x)\leq L_f\|x -
        q_i\|^2\label{eq:approx_major_interim2}.
    \end{align}
    We obtain an upper bound on the approximation errors $f
    - f_t^-$ and $f^+_t - f$ by computing the finite minimum
    of \eqref{eq:approx_minor_interim2} and
    \eqref{eq:approx_major_interim2} over $i\in\Nint{t}$.
\end{proof}

Using \eqref{eq:strong_cvx}, a $\mu$-convex $f$ admits a
tighter, piecewise-quadratic, data-driven
minorant $\ftsc: \mathcal{X} \to \mathbb{R}$,
\begin{align}
    \ftsc(x)&=\max_{i\in\Nint{t}} \left(\ell(x;q_i, f) + \frac{\mu}{2}\|x -
q_i\|^2\right)\label{eq:approx_sc},
\end{align}
with $f_t^-\leq \ftsc \leq f$.

\subsection{Covering method for global optimization of
\eqref{prob:orig_prob_noH}}
\label{sub:cov_no_H}

We briefly discuss how covering method solves $\inf_{x\in
\mathcal{X}} J(x)$, which motivates the proposed algorithms. 
\begin{algorithm}%
\caption{Covering method for global optimization of
\eqref{prob:orig_prob_noH}~\cite{blanquero_covering_2000}} 
    \label{algo:cov_method}
\KwIn{Convex \& compact set $ \mathcal{X}\subset
    \mathbb{R}^d$, first-order oracle for $J$,
    initial point $q_1\in \mathcal{X}$, Lipschitz
    gradient constant $L_J$, suboptimality threshold
$\eta > 0$}
    \KwOut{Near-global minima of \eqref{prob:orig_prob_noH} with
    suboptimality bound}

    Initialize $\xddagglobal \gets q_1$ and $\Deltaglobal \gets
    \infty$, and query the first-order oracle at $q_{1}$

    \For{$t=1,2,3,\ldots$}
    {
        Construct $J_t^-$ using \eqref{eq:minor_f}

        Solve the following optimization problem to compute
        $q_{t+1}$ and $J_t^-(q_{t+1})$
        \begin{align}
            q_{t+1}\gets\arginf\nolimits_{x}\ J_t^-(x)\ \
            \text{subject to}\ x\in
            \mathcal{X}\label{prob:cov_method}
        \end{align}

        Query the first-order oracle at $q_{t+1}$

        Update the near-global minima estimate and
        suboptimality bound:
        \begin{align}
            \xddagglobal&\gets\argmin\nolimits_{\{q_i:
            i\in\Nint{t+1}\}}
            J(q_i)\label{eq:cov_method_min}\\
            \Deltaglobal&\gets\hspace*{1.8em}
            \min\nolimits_{\{q_i: i\in\Nint{t+1}\}}
            J(q_i)-J_t^-(q_{t+1})\label{eq:cov_method_subopt}
        \end{align}

        \textbf{if }$\Deltaglobal \leq \eta$\textbf{ then break
        }\hfill\emph{$\triangleright$ Terminate,
        if acceptable global suboptimality bound}

    }

    \KwRet ($\xddagglobal,\
    \Deltaglobal$)\hfill\emph{$\triangleright\ \Deltaglobal <\infty$ at the
    end of first iteration}
\end{algorithm}

Algorithm~\ref{algo:cov_method} computes $\xddagglobal\in
\mathcal{X}$ and $\Deltaglobal \geq 0$ such that 
\begin{align}
    0\leq J(\xddagglobal) - J(\xopt)\leq
    \Deltaglobal.\label{eq:cov_prob_desirata}
\end{align}
At each iteration of Algorithm~\ref{algo:cov_method}, the
optimization problem \eqref{prob:cov_method} is feasible and
has a finite optimal solution since $\mathcal{X}$ is compact
and non-empty.  Furthermore, $\xddagglobal$ satisfies
\eqref{eq:cov_prob_desirata} at every iteration $t\in\Npos$, since
\begin{align}
    J_t^-(q_{t+1})\leq J_t^-(\xopt)\leq J(\xopt)\leq
J(\xddagglobal) \triangleq \min_{i\in\Nint{t+1}}
J(q_i).\label{eq:cov_prob_ineq}
\end{align} 
Equation \eqref{eq:cov_prob_ineq} follows from
\eqref{eq:minor_f}, \eqref{prob:cov_method}, and the fact
that $\xddagglobal$ is always feasible for
\eqref{prob:orig_prob_noH}.
See~\cite{blanquero_covering_2000,fortin_piecewise_2011,pardalos2010deterministic}
for more details.

To characterize the upper limit on the number of iterations required to ensure
that the global-suboptimality bound is below $\eta$, we
first recall Lemma~\ref{lem:pidgeonhole}, which follows from
the \emph{pidgeonhole principle}. 

\begin{lem}\label{lem:pidgeonhole}
\newcommand{\lemPidgeonHole}{ For any $\epsilon > 0$, any
    convex and compact set $ \mathcal{X} \subset
    \mathbb{R}^d$, and any finite collection of $T\geq
    \left\lceil{\left({\diam(\mathcal{X})\sqrt{d}}\right)}^d\epsilon^{\frac{-d}{2}}\right\rceil
    + 1$ distinct points $q_i\in \mathcal{X}$ for every
    $i\in \Nint{T}$, there exists $t\in \Nint{T-1}$ such
    that $\min_{i\in\Nint{t}}\|q_{t+1} - q_i\|^2\leq
\epsilon$.}\lemPidgeonHole{}
\end{lem}
\begin{proof}
    The set $\mathcal{X}$ is covered by a
    hypercube of side $\diam(\mathcal{X})$.
    For any $\epsilon>0$
    The minimum number of hypercubes of side
    $\sqrt{\frac{\epsilon}{d}}$ that covers the hypercube of
    side $\diam(\mathcal{X})$
    is given by $\left\lceil {\frac{\diam(
            \mathcal{X})^d}{\left(\epsilon/d\right)^\frac{d}{2}}
    }\right\rceil$. Note that $T$ is at least one more than
    this minimum number. By the \emph{pidgeonhole principle}, at
    least one of the hypercubes with side
    $\sqrt{\frac{\epsilon}{d}}$ must have at least two
    points. However, the maximum separation allowed between two
    points within such a hypercube is $\sqrt{\epsilon}$. Thus,
    for some $i,j\in\Nint{T}$ with $i\neq j$, we have
    $\|q_i - q_j\|\leq \sqrt{\epsilon}$. We complete the
    proof with $t\triangleq\max(i,j)-1\in\Nint{T-1}$.
\end{proof}
\begin{prop}[\newcommand{\propBoundOnTAlgoOneHeading}{\textsc{Worst-case, sufficient budget for
    Algorithm~\ref{algo:cov_method}}}\propBoundOnTAlgoOneHeading{}]\label{prop:BoundOnTAlgoOne}
    \newcommand{\propBoundOnTAlgoOne}{For a user-specified
        suboptimality threshold $\eta >0$, define
        $\Tsuffcov=\left\lceil{(\diam(\mathcal{X})\sqrt{d})^d
        \left({\frac{L_J}{\eta}}\right)^\frac{d}{2}}\right\rceil
        + 1$.  Algorithm~\ref{algo:cov_method} terminates at
        an iteration $t\leq \Tsuffcov$ satisfying
        \eqref{eq:cov_prob_desirata} with
        $\Deltaglobal\leq\eta$, irrespective of the
        choice of $J\in\FJ$ and the initial solution $q_1\in
        \mathcal{X}$.}\propBoundOnTAlgoOne{}
\end{prop}
\begin{proof} 
    From
    \eqref{eq:cov_method_subopt} and Lemma~\ref{lem:approx},
    we have for every $t\in\Nint{T-1}$,
    \begin{align}
        \Deltaglobal=J(\xddagglobal) - J_{t}^-(q_{t+1})\leq J(q_{t+1}) - J_{t}^-(q_{t+1})\leq L_J
        \min_{i\in\Nint{t}} \|q_t -
        q_i\|^2. \label{eq:contradiction_step}
    \end{align}
    The proof follows from Lemma~\ref{lem:pidgeonhole} and
    the choice of the budget $T$.
\end{proof}

Proposition~\ref{prop:BoundOnTAlgoOne} shows that
Algorithm~\ref{algo:cov_method} is \emph{worst-case optimal}
for the class of smooth optimization problems of the form
\eqref{prob:orig_prob_noH} with twice-differentiable
objective function $J$. For $L_J$ set to the true Lipschitz
gradient constant $K_J$, the sufficient budget $\Tsuffcov$
prescribed for Algorithm~\ref{algo:cov_method} by
Proposition~\ref{prop:BoundOnTAlgoOne} matches the
well-known minimum number of iterations necessary for any
first-order sequential optimization algorithm to solve
\eqref{prob:orig_prob_noH}~\cite[Thm.
4]{vavasis1995complexity}, up to a constant independent of
$K_J$ and $\eta$.

\section{Tractable algorithms for global optimization of
\eqref{prob:orig_prob}}
\label{sec:main}

We now present the main results, and address
Problems~\ref{p_st:algo} and~\ref{p_st:algo_safe}.
Specifically, we propose Algorithms~\ref{algo:my_cov_method}
and~\ref{algo:my_cov_method_mu_convex} (see
page~\pageref{algo:my_cov_method}) for the global
optimization of \eqref{prob:orig_prob}. These first-order
sequential optimization algorithms
solve tractable, surrogate optimization problems of
\eqref{prob:orig_prob}, constructed using the past oracle
queries and smoothness information, to determine the
next query point. 
Algorithm~\ref{algo:my_cov_method_mu_convex} achieves constraint
violation-free optimization by performing an
additional projection step.

To help the reader put the proposed 
algorithms in context, we provide two illustrative
examples in page~\pageref{tab:illustration}.  The first
example demonstrates how the choice of hyperparameter $L_J$
in Algorithm~\ref{algo:my_cov_method} affects the
approximants, and consequently, the number of iterations to
solve \eqref{prob:orig_prob} and the number of constraint
violations.  The second example shows that
Algorithm~\ref{algo:my_cov_method_mu_convex} can compute a
solution to \eqref{prob:orig_prob}, without incurring any
constraint violation, when $H$ is additionally known 
to be strongly-convex. Both of the algorithms escape a
local minimum near the initial solution guess $q_1$ to
arrive at the global minimum. We provide the numerical
details of the examples in Appendix~\ref{app:example}. 

\subsection{Global optimization of \eqref{prob:orig_prob}
for smooth (possibly non-convex) $H$}
\label{sub:CMOE}

Algorithm~\ref{algo:my_cov_method} solves
\eqref{prob:orig_prob} with smooth (possibly non-convex)
functions $J$ and $H$. It constructs iterates
by solving a surrogate optimization problem
\eqref{prob:relax_surrogate_prob}. Note that every feasible solution
of \eqref{prob:orig_prob} is feasible for
\eqref{prob:relax_surrogate_prob}, since the constraint
$H_t^-\leq 0$ is a relaxation of the constraint $H\leq 0$ in
\eqref{prob:orig_prob} (Lemma~\ref{lem:approx}). Motivated
by Algorithm~\ref{algo:cov_method},
Algorithm~\ref{algo:my_cov_method} also replaces the unknown
objective $J$ of \eqref{prob:orig_prob} with its known,
data-driven minorant $J_t^-$ in \eqref{prob:relax_surrogate_prob}.
Algorithm~\ref{algo:my_cov_method} does not require a
feasible initial solution guess $q_1\in \mathcal{X}$ to
solve \eqref{prob:orig_prob}.

We prove the correctness of
Algorithm~\ref{algo:my_cov_method} using a
$\delta$-approximation of \eqref{prob:orig_prob},
\begin{align}
    \text{$\delta$-relax of
    \eqref{prob:orig_prob}}:\ \ \mathrm{minimize}\ J(x)\quad
    \mathrm{subject\ to}\ x \in
    \mathcal{X},\ H(x)\leq
    \delta\label{prob:orig_prob_relax_delta}.
\end{align}
The optimal values of \eqref{prob:orig_prob}
and \eqref{prob:orig_prob_relax_delta} are closely related
under the following assumption.

\begin{figure*}[p]
    \input{additional_texfiles/algorithms_my_cov_method.tex}
\end{figure*}
\newcommand{\shiftingUpForText}{1.5em}
\begin{table*}[p]
    \newcommand{\trimValues}{0 0 0 0}
    \centering
    \begin{tabular}{c c}
        \hspace*{-11em}\begin{minipage}[b]{0.2\linewidth}
            \centering
            Example 1: Non-convex $H$, infeasible
            start\vspace*{\shiftingUpForText}
        \end{minipage} 
       &\hspace*{-11.75em}
       \includegraphics[width=0.5\linewidth,Trim=\trimValues, clip]{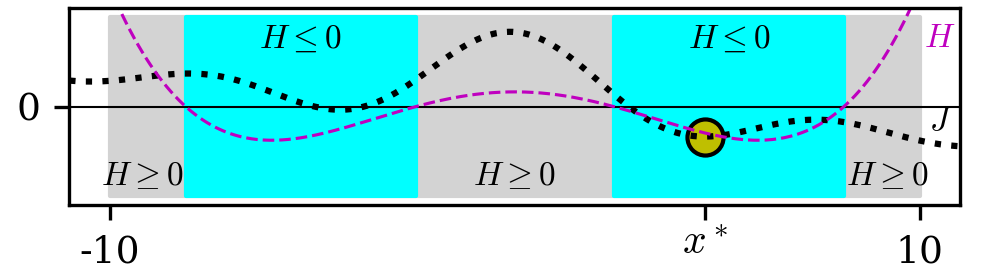}
       \begin{minipage}[b]{0.25\linewidth}
           \centering
            See Appendix~\ref{app:example} for\linebreak more
            details on the example
            problem.\vspace*{\shiftingUpForText}
        \end{minipage}\\[0.5ex]
    \includegraphics[width=0.45\linewidth,Trim=\trimValues, clip]{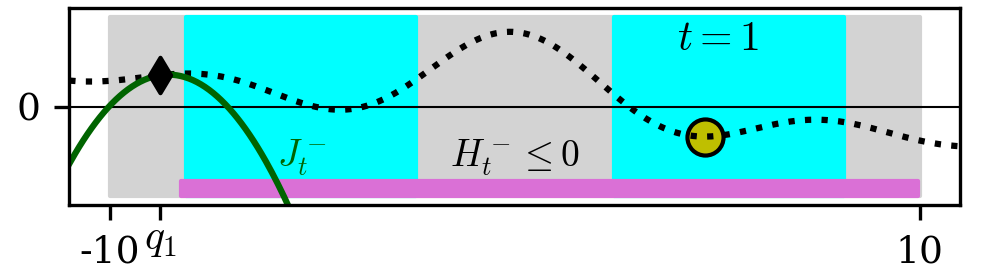} & 
    \includegraphics[width=0.45\linewidth,Trim=\trimValues, clip]{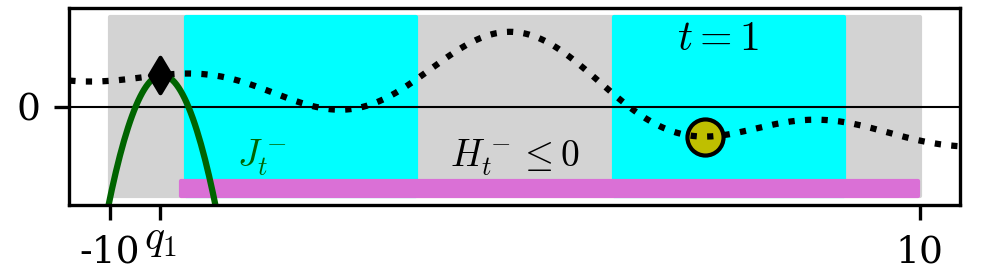}\\
    \includegraphics[width=0.45\linewidth,Trim=\trimValues, clip]{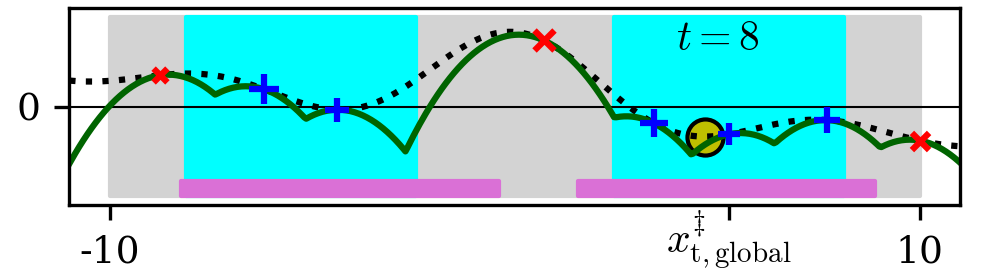} &
    \includegraphics[width=0.45\linewidth,Trim=\trimValues, clip]{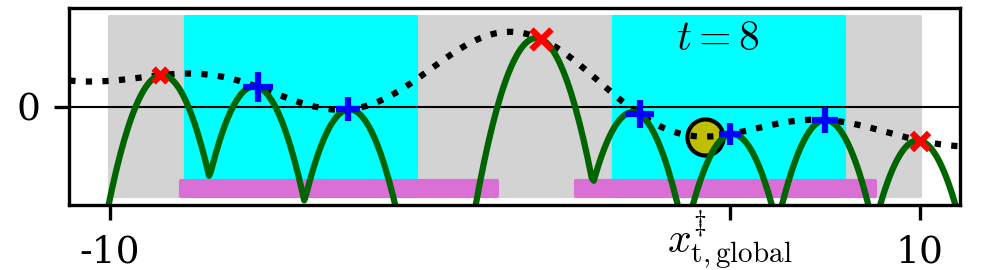}\\
    \includegraphics[width=0.45\linewidth,Trim=\trimValues, clip]{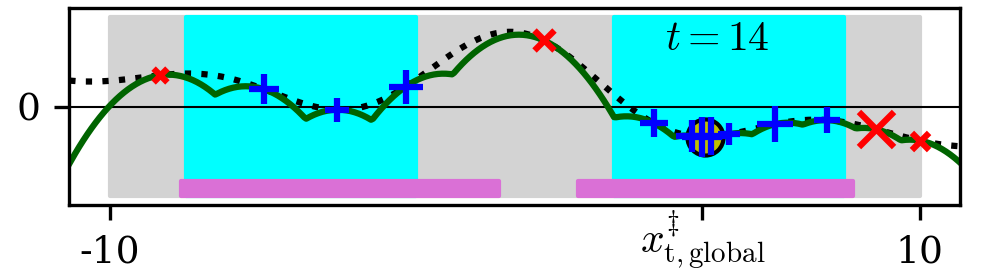} &
    \includegraphics[width=0.45\linewidth,Trim=\trimValues, clip]{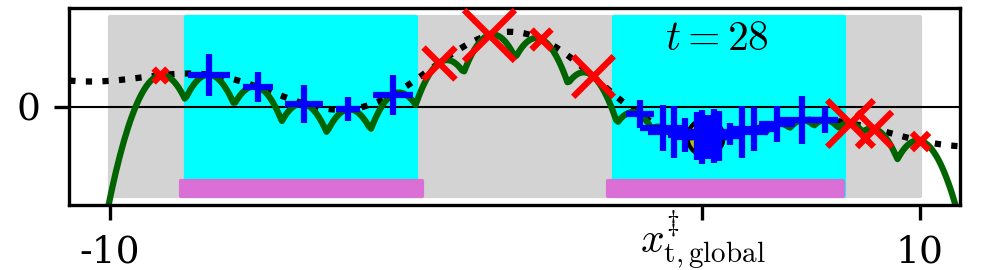}\\
    \begin{minipage}[t]{0.45\linewidth}
        \centering
        Algorithm~\ref{algo:my_cov_method} with $L_J=0.2$ \linebreak 
        computes an $(0,\eta, \delta)$-minimum in \linebreak 
        $14$ iterations with $4$ infeasible queries.\\
    \end{minipage} & 
    \begin{minipage}[t]{0.45\linewidth}
        \centering
        Algorithm~\ref{algo:my_cov_method} with $L_J = 5
        \times 0.2$\linebreak computes an $(0,\eta, \delta)$-minimum in\linebreak
        $28$ iterations with $7$ infeasible queries.\\
    \end{minipage} \\[3em]
    \hspace*{-11em}\begin{minipage}[b]{0.2\linewidth}
            \centering
            Example 2: $\mu$-convex $H$, feasible
            start\vspace*{\shiftingUpForText}
        \end{minipage} 
       &\hspace*{-11.75em}
       \includegraphics[width=0.5\linewidth,Trim=\trimValues, clip]{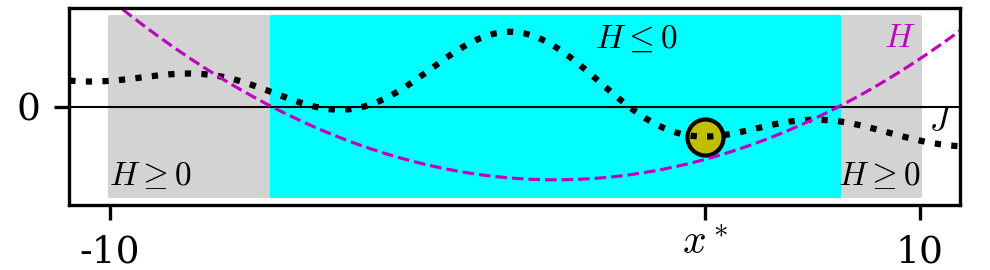}
       \begin{minipage}[b]{0.25\linewidth}
           \centering
            See Appendix~\ref{app:example} for\linebreak more
            details on the example
            problem.\vspace*{\shiftingUpForText}
        \end{minipage}\\[0.5ex]
    \includegraphics[width=0.45\linewidth,Trim=\trimValues, clip]{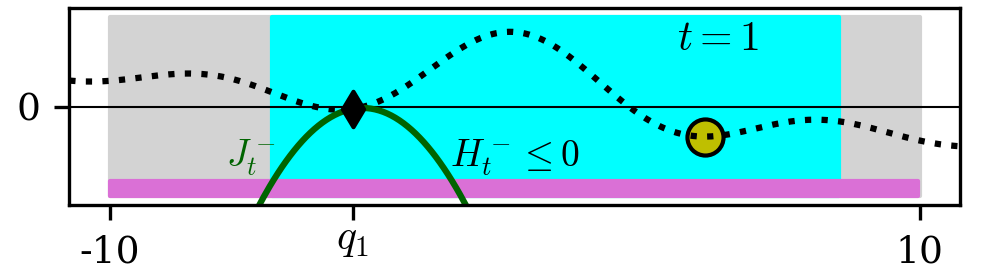} &
    \includegraphics[width=0.45\linewidth,Trim=\trimValues, clip]{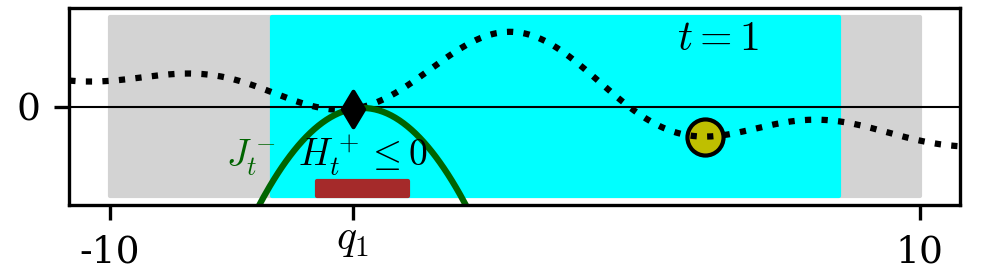}\\
    \includegraphics[width=0.45\linewidth,Trim=\trimValues, clip]{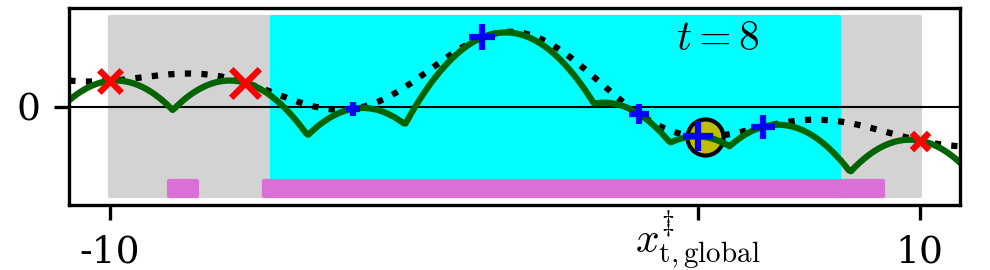} &
    \includegraphics[width=0.45\linewidth,Trim=\trimValues, clip]{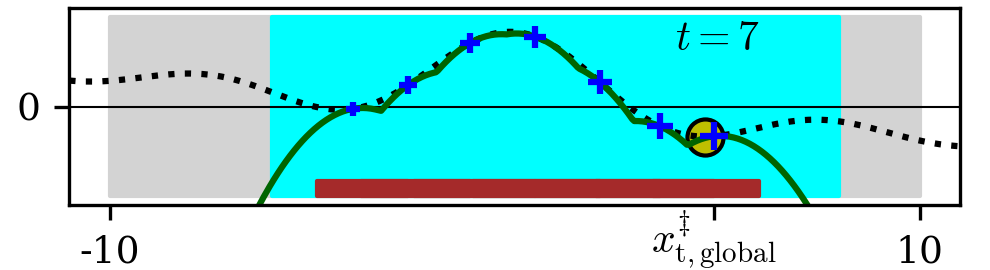}\\
    \includegraphics[width=0.45\linewidth,Trim=\trimValues, clip]{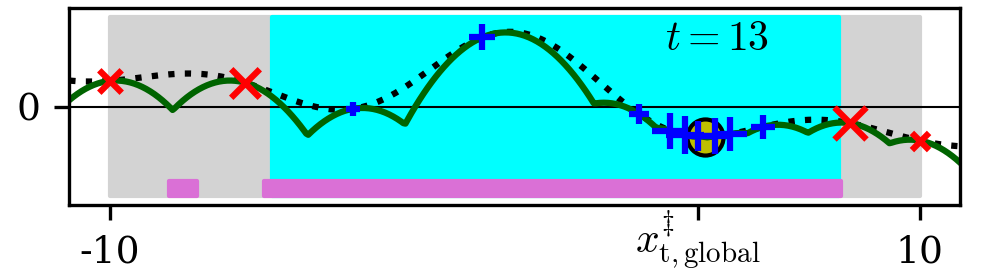} &
    \includegraphics[width=0.45\linewidth,Trim=\trimValues, clip]{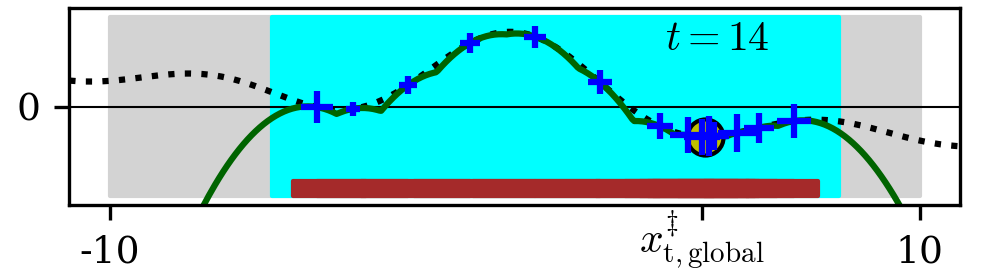}\\
    \begin{minipage}[t]{0.45\linewidth}
        \centering
        Algorithm~\ref{algo:my_cov_method} computes an\linebreak
        $(0,\eta,\delta)$-minimum in $13$ iterations\linebreak 
        with $4$ infeasible queries.\\
    \end{minipage} & 
    \begin{minipage}[t]{0.45\linewidth}
        \centering
        Algorithm~\ref{algo:my_cov_method_mu_convex} computes an\linebreak
        $(0,\eta,0)$-minimum in $14$ iterations\linebreak 
        with $0$ infeasible queries.\\
    \end{minipage} \\[3.5em]
   \end{tabular}
   \vspace*{-1em}
   \caption{Illustration of Algorithms~\ref{algo:my_cov_method}
       and~\ref{algo:my_cov_method_mu_convex} with
       $\eta=0.01$ and $\delta=10^{-5}$. (top-left and top-right) Algorithm~\ref{algo:my_cov_method} needs more
       iterations to solve \eqref{prob:orig_prob} for larger
       $L_J$, while handling non-convex $H$ and infeasible
       start. (bottom-left and bottom-right)
       For a $\mu$-convex $H$ and feasible start,
       Algorithm~\ref{algo:my_cov_method_mu_convex} computes
       near-global optimum without any infeasible queries.
       Algorithm~\ref{algo:my_cov_method} constructs a
       monotone decreasing sequence of $\{H_t^-\leq 0\}$
       such that $\xopt\in\{H_t^-\leq 0\}$ at every
       iteration, while Algorithm~\ref{algo:my_cov_method_mu_convex} constructs a
       monotone increasing sequence of $\{H_t^+\leq 0\}$
       such that $\xopt\in\{H_t^+\leq 0\}$ at some iteration.
       Legend: initial guess $q_1$ ($\blacklozenge$),
       feasible ({\color{blue} $+$}) and infeasible
       ({\color{red} x}) queries for \eqref{prob:orig_prob}.}\label{tab:illustration}
\end{table*}

\begin{assum}[\textsc{Well-behaved $H$ at the boundary of
    $\{H\leq 0\}$}]\label{assum:good_grad_my_cov_method}
    The constraint function $H$ and the relaxation
    threshold $\delta > 0$ satisfies $ \|\nabla H(x)\|
    \geq \sqrt{2L_H\delta}$ for every $x\in \{0< H\leq
    \delta\}$. 
\end{assum}
Assumption~\ref{assum:good_grad_my_cov_method} ensures that
the gradient of $H$ over the ``excess'' feasible space, the
set $\{0< H\leq \delta\}$ arising from
relaxation of the constraint $\{H\leq 0\}$ to $\{H\leq
\delta\}$ is bounded away from zero
(Figure~\ref{fig:assumption2} on
page~\pageref{fig:assumption2}). The requirement on $H$
imposed by Assumption~\ref{assum:good_grad_my_cov_method}
weakens as the user-specified relaxation threshold $\delta$
approaches zero.
\begin{figure}[h]
    \centering
    \includegraphics[width=0.95\linewidth]{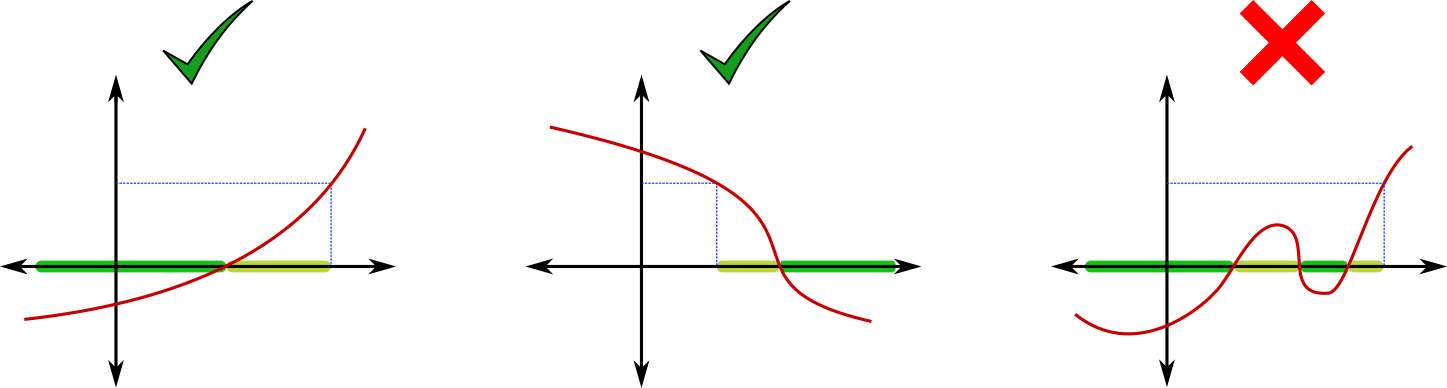}
    \begin{tikzpicture}[overlay, remember picture]
        \small
        \node[color=black] at (-9.15, 2.15) {$H$};
        \node[color=black] at (-10,0.75) {$0 < H\leq \delta$};
        \node[color=black] at (-12.3,1.3) {$H\leq 0$};
        \node at (-11.95,1.75) {$\delta$};
        \node at (-7.45,1.75) {$\delta$};
        \node at (-2.95,1.75) {$\delta$};
    \end{tikzpicture}
    \caption{Assumption~\ref{assum:good_grad_my_cov_method}
        avoids pathological cases, like the rightmost
        figure, that may arise when using a large user-specified relaxation
    threshold $\delta>0$.}
    \label{fig:assumption2}
\end{figure}

\begin{prop}[\textsc{$\delta$-relaxation
    \eqref{prob:orig_prob_relax_delta} closely approximates
\eqref{prob:orig_prob}}]\label{prop:my_cov_method_relation}
Let Assumption~\ref{assum:good_grad_my_cov_method} hold.
Then, whenever \eqref{prob:orig_prob} is feasible, the
global minimum $\xoptdelta\in \mathcal{X}$ of
\eqref{prob:orig_prob_relax_delta} is related to the global
minimum $\xopt$ of \eqref{prob:orig_prob}, in the following
sense 
    \begin{align}
        0\leq J(\xopt) - J(\xoptdelta) \leq
        \GJmax\sqrt{\frac{2\delta}{L_H}}.\label{eq:relationship}
    \end{align}
\end{prop}
\begin{proof}
    Note that the
    feasible set $ \mathcal{X}\cap \{H\leq \delta\}$ of
    \eqref{prob:orig_prob_relax_delta} is non-empty and
    bounded, and contains the non-empty feasible solution space of \eqref{prob:orig_prob}. Therefore, $\xoptdelta$ exists, and we
    trivially have the lower bound $J(\xopt) - J(\xoptdelta)
    \geq 0$.

    To prove the proposed upper bound on $J(\xopt) -
    J(\xoptdelta)$, we study two cases --- $H(\xoptdelta)
    \leq 0$ and $0 < H(\xoptdelta) \leq \delta$. In either
    cases, we will construct a feasible solution $z$
    for \eqref{prob:orig_prob}, and characterize an upper
    bound for $J(z)$ using $J(\xoptdelta)$ to complete the
    proof. For the first case ($H(\xoptdelta) \leq 0$), we
    choose $z = \xoptdelta$ to trivially satisfy the upper bound. 

    In the second case, define $z = \xoptdelta -
    \frac{2\delta}{\|\nabla H(\xoptdelta)\|^2} \nabla
    H(\xoptdelta)$. Here, $z$ is well-defined since $\|\nabla
    H(\xoptdelta)\| \geq \sqrt{2L_H\delta}$ in this case by
    Assumption~\ref{assum:good_grad_my_cov_method}. We have
    $H(z)\leq 0$ and consequently feasible for
    \eqref{prob:orig_prob}, since
    \begin{alignat}{2}
        H(z) &\leq H(\xoptdelta) + \nabla H(\xoptdelta)\cdot
        (z-\xoptdelta) + \frac{L_H}{2} \| z -\xoptdelta
        \|^2,
             &\quad&\mbox{(By
             \eqref{eq:approx_major_interim})}\nonumber \\
            &= H(\xoptdelta) + \nabla H(\xoptdelta)\cdot
            \left({-
    \frac{2\delta\nabla
    H(\xoptdelta)}{\|\nabla H(\xoptdelta)\|^2} }\right) + \frac{L_H}{2} \| z -\xoptdelta
    \|^2, &&\mbox{(By the choice of $z$)}\nonumber \\
    &= H(\xoptdelta) - 2 \delta +  \frac{L_H}{2} \frac{4
    \delta^2}{\|\nabla H(\xoptdelta)\|^2} 
    \leq \delta - 2 \delta + \frac{2\delta^2}{2\delta} = 0.
    \nonumber
    \end{alignat}
    Here, we have used the observation that
    $\frac{L_H}{\|\nabla H(\xoptdelta)\|^2}\leq
    \frac{1}{2\delta}$ since $\|\nabla H(\xoptdelta)\| \geq
    \sqrt{2L_H \delta}$.

    Using \eqref{eq:mvt_cs} on $J$ and the upper bound on
    $\|z - \xoptdelta\|= \frac{2\delta}{\|\nabla
    H(\xoptdelta)\|} \leq \sqrt{\frac{2 \delta}{L_H}}$ by
    the choice of $z$, we have $|J(z) -
    J(\xoptdelta)|\leq\GJmax\sqrt{\frac{2\delta}{L_H}}$.
    Therefore, $J(\xoptdelta)\leq J(\xopt) \leq J(z) \leq
    J(\xoptdelta) + \GJmax\sqrt{2\delta/L_H}$, as desired.
\end{proof}

\begin{thm}[\textsc{Algorithm~\ref{algo:my_cov_method} addresses
    Problem~\ref{p_st:algo}}]\label{thm:my_cov_method}
    For any objective function $J$,
    constraint function $H$, and initial solution
    $q_1$, the following statements about
    Algorithm~\ref{algo:my_cov_method} are true for any
    relaxation threshold $\delta$ that satisfies
    Assumption~\ref{assum:good_grad_my_cov_method}:
    {\renewcommand{\theenumi}{\Alph{enumi}}
    \begin{enumerate}
        \item \textsc{(Well-definedness)} For any iteration
            of Algorithm~\ref{algo:my_cov_method}, the
            optimization problem
            \eqref{prob:relax_surrogate_prob} either has a
            finite optimal solution or generates a proof of
            infeasibility of \eqref{prob:orig_prob}
            (resulting in the termination of
            Algorithm~\ref{algo:my_cov_method} at
            Step~\ref{step:algo_my_cov_method_infeas}).\label{thm:my_cov_method_feas}
        \item \textsc{(Anytime property)} Let
            Algorithm~\ref{algo:my_cov_method} run up to an
            iteration $t\in \mathbb{N}$, without arriving at
            a proof of infeasibility of
            \eqref{prob:orig_prob}
            (Theorem~\ref{thm:my_cov_method}\ref{thm:my_cov_method_feas}).
            {
            \begin{enumerate}
                \item\label{thm:my_cov_method_anytime_finite_delta}  
            If $\Deltaglobal<\infty$, then $\xddagglobal$ is a 
            $\left(\GJmax\sqrt{\frac{2\delta}{L_H}},\Deltaglobal,\delta\right)$-minimum
            of \eqref{prob:orig_prob}. Additionally, if
            $\xddagglobal$ is feasible for
            \eqref{prob:orig_prob}, then $\xddagglobal$ is a
            $\left(0,\Deltaglobal,\delta\right)$-minimum of
            \eqref{prob:orig_prob}. 
                \item 
            \label{thm:my_cov_method_anytime_infinite_delta}
            Else ($\Deltaglobal$ is $\infty$), 
            \eqref{prob:orig_prob} is $\gamma$-infeasible,
$\gamma\triangleq-\min(\inf\nolimits_{x \in \mathcal{X}}
H_t^-(x),0)$.  \end{enumerate}}
        \item \textsc{(Worst-case, sufficient
            budget)} For a suboptimality threshold
            $\eta \geq \GJmax\sqrt{\frac{2\delta}{L_H}}$,
            define
            \begin{align}
                \Tsuff&\triangleq \left\lceil{(\diam(
                                \mathcal{X})\sqrt{d})^d\left({{\left(\frac{L_J}{\eta}\right)}^\frac{d}{2}
                                +
                {\left(\frac{L_H}{\delta}\right)}^\frac{d}{2}}\right)}\right\rceil
                        + 1. \nonumber
            \end{align}
            Then, Algorithm~\ref{algo:my_cov_method}
            terminates with an
            $(\eta,\delta)$-minimum of
            \eqref{prob:orig_prob} or declares
            \eqref{prob:orig_prob} is infeasible at some
            iteration $t\leq
            \Tsuff$.\label{thm:my_cov_method_tsuff}
    \end{enumerate}
    }
\end{thm}
\begin{proof}
    \emph{Proof of~\ref{thm:my_cov_method_feas})} 
    By
    Lemma~\ref{lem:approx}, every
    feasible solution of \eqref{prob:orig_prob} is feasible
    for \eqref{prob:relax_surrogate_prob} at every iteration
    of Algorithm~\ref{algo:my_cov_method} since $H_t^-\leq
    H$ for every iteration $t$. Therefore, the 
    infeasibility of \eqref{prob:relax_surrogate_prob} at
    any iteration implies the infeasibility of 
    \eqref{prob:orig_prob}. On the other hand, due to the
    compactness of $ \mathcal{X}$, the optimal
    solution of \eqref{prob:relax_surrogate_prob} is finite
    whenever \eqref{prob:relax_surrogate_prob} is feasible.
    
    \emph{Proof
    of~\ref{thm:my_cov_method_anytime_finite_delta})} Since
    $\Deltaglobal<\infty$ at the end of iteration
    $t\in\Nint{T-1}$, we know that there is some
    iteration $i\in [t]$ such that $H(q_{i+1})\leq \delta$.
    Consequently, $\xddagglobal$ is well-defined, and
    $H(\xddagglobal) \leq \delta$. 
    From 
    \eqref{prob:orig_prob_relax_delta} and
    \eqref{eq:my_cov_method_min}, 
    $J(\xoptdelta)\leq J(\xddagglobal)\triangleq
    \min_{\mycovmethodqi} J(q_i)$. We also know that
    $J_{t}^-(q_{t+1})\leq J_{t}^-(\xoptdelta)$, since
    $q_{t+1}$ is the optimal solution of 
    \eqref{prob:relax_surrogate_prob} at iteration $t$.
    From Lemma~\ref{lem:approx}, we have
    \begin{align}
        J_{t}^-(q_{t+1})\leq J_{t}^-(\xoptdelta)\leq J(\xoptdelta)\leq J(\xddagglobal). \label{eq:mycovmethod_ineq}
    \end{align}
    The desired bounds on the true global suboptimality
    $J(\xddagglobal) - J(\xopt)$ at the end of iteration $t$
    follows from \eqref{eq:mycovmethod_ineq} and
    Proposition~\ref{prop:my_cov_method_relation}.
    Specifically, 
    \begin{alignat}{2}
        J(\xddagglobal) - J(\xopt)&\leq J(\xddagglobal) -
        J(\xoptdelta)&&\leq
        \Deltaglobal, \nonumber \\
        J(\xddagglobal) - J(\xopt)&\geq J(\xddagglobal) -
        J(\xoptdelta) - \GJmax\sqrt{\frac{2\delta}{L_H}} &&\geq -
        \GJmax\sqrt{\frac{2\delta}{L_H}}, \nonumber
    \end{alignat}
    where $\Deltaglobal$ is defined in
    \eqref{eq:my_cov_method_subopt}. We have a tighter lower
    bound on the global suboptimality $J(\xddagglobal) -
    J(\xopt) \geq 0$, when $H(\xddagglobal) \leq 0$.

    \emph{Proof
    of~\ref{thm:my_cov_method_anytime_infinite_delta})}
    Since $\Deltaglobal$ is $\infty$ at the end of iteration
    $t\in\Nint{T-1}$, there is no query $q_i$ with
    $i\in\Nint{t+1}$ such that $H(q_i)\leq \delta$. Consequently, 
    $\xddagglobal$ is not defined. On other hand, we know
    that $\inf_{x\in \mathcal{X}} H_t^-(x) \leq 0$ since
    \eqref{prob:relax_surrogate_prob} was feasible at
    iteration $t$. By Lemma~\ref{lem:approx}, we know that
    $H_t^-\leq H$, which implies that \eqref{prob:orig_prob}
    is $(-\inf_{x\in \mathcal{X}} H_t^-(x))$-infeasible. We
    define $\gamma=-\min(\inf_{x\in \mathcal{X}}
    H_t^-(x)),0)$ to ensure that $\gamma \geq 0$.
    
    \emph{Proof of~\ref{thm:my_cov_method_tsuff})} 
    For the given budget $T$, define
    $\Pi\triangleq\{i\in\Nint{T-1}: H(q_i) \leq
    \delta\}\subseteq \Nint{T-1}$ as the finite set of
    iterations in Algorithm~\ref{algo:my_cov_method} that
    resulted in feasible iterates for
    \eqref{prob:relax_surrogate_prob}.
    Algorithm~\ref{algo:my_cov_method} has three outcomes
    for the surrogate optimization problem
    \eqref{prob:relax_surrogate_prob}:
    \begin{enumerate}
        \item \eqref{prob:relax_surrogate_prob} is
            infeasible at some iteration $t\in\Nint{T-1}$,
        \item \eqref{prob:relax_surrogate_prob} is
            feasible at every iteration $t\in\Nint{T-1}$, but
            $\Pi$ is empty, or
        \item \eqref{prob:relax_surrogate_prob} is
            feasible at every iteration $t\in\Nint{T-1}$, and
            $\Pi$ is non-empty.
    \end{enumerate}
    By
    Theorem~\ref{thm:my_cov_method}\ref{thm:my_cov_method_feas},
    the first outcome will result in
    Algorithm~\ref{algo:my_cov_method} terminating with a
    proof of infeasibility.
    Therefore, we need to only focus on the second and the
    third outcomes.

    Next, we rule out the second outcome when $T\geq
    T_\delta\triangleq\left\lceil{(\diam(
    \mathcal{X})\sqrt{d})^d
    {\left(\frac{L_H}{\delta}\right)}^\frac{d}{2}}\right\rceil
    + 1$. Assume, for contradiction, that the second outcome
    occurs when solving some problem instance of
    \eqref{prob:orig_prob} using
    Algorithm~\ref{algo:my_cov_method} with $T=T_\delta$.
    Specifically, we have assumed that 
    \eqref{prob:relax_surrogate_prob} is feasible at every
    iteration, but the set $\Pi$ is empty.  Consequently,
    for every iteration $t\in\Nint{T-1}$,
    $H_t^-(q_{t+1})\leq 0$ ($q_{t+1}$ is feasible for
    \eqref{prob:relax_surrogate_prob}) and $H(q_{t+1})>
    \delta$ ($\Pi$ is empty). Therefore, 
    \begin{align}
        H_t^-(q_{t+1})\leq 0 < \delta <H(q_{t+1})
        \Longrightarrow \delta < H(q_{t+1}) -
        H_t^-(q_{t+1}),\label{eq:contradiction_step_tsuff}
    \end{align}
    for every iteration $t\in\Nint{T-1}$. Recall that
    $H(q_{t+1}) -
    H_t^-(q_{t+1})\leq L_H\min_{i\in\Nint{t}}\|q_{t+1} -
    q_i\|^2$ at every iteration $t\in\Nint{T-1}$ by
    Lemma~\ref{lem:approx}. We obtain a contradiction of
    \eqref{eq:contradiction_step_tsuff}, as desired, via
    Lemma~\ref{lem:pidgeonhole} and the choice of
    $T_\delta$. 
    
    Having ruled out the second outcome for $T= \Tsuff >
    T_\delta$, we now turn to the third outcome. We claim
    that, in this case, \emph{Algorithm~\ref{algo:my_cov_method}
    terminates with an $(\eta,\delta)$-minimum of
\eqref{prob:orig_prob} at some iteration $t\in[\Tsuff-1]$}.
    To prove the claim, we again pursue a proof via
    contradiction. Specifically, we assume for
    contradiction that \eqref{prob:relax_surrogate_prob} is
    feasible at every iteration $t\in\Nint{\Tsuff-1}$, the
    set $\Pi$ is non-empty, and $\Deltaglobal > \eta$. 
    By the same arguments used to rule out the second
    outcome, 
    $|\Pi| \geq \Tsuff-(T_\delta-1)$. From
    Lemma~\ref{lem:approx}, we arrive at the contradiction
    that for every $t\in\Pi$, $\Deltaglobal\triangleq
    J(\xddagglobal) - J_t^-(q_{t+1}) \leq J(q_{t+1}) -
    J_t^-(q_{t+1})\leq L_J\min_{i\in\Nint{t}} \|q_{t+1} -
    q_i\|^2\leq L_J\min_{i\in\Pi} \|q_{t+1} -
    q_i\|^2 \leq \eta$ by Lemma~\ref{lem:pidgeonhole} and
    the choice of $\Tsuff$.  
\end{proof}

We now briefly discuss a minor modification to 
Algorithm~\ref{algo:my_cov_method} that can significantly
improve the computed near-infeasibility certificate
$\gamma$. The modification is relevant only when the
feasibility of \eqref{prob:orig_prob} is unknown and the
initial solution guess $q_1$ is infeasible. Instead of
determining $\gamma$ based on the iterates obtained from
solving \eqref{prob:relax_surrogate_prob}
(Step~\ref{step:algo_my_cov_method_infeas} of
Algorithm~\ref{algo:my_cov_method}), we design oracle
queries based on the following rule
\begin{align}
    q_{t+1} = \inf_{x\in \mathcal{X}} H_t^-(x),\label{eq:new_infeas_rule}
\end{align}
until either $H(q_{t+1}) \leq \delta$, which ensures that
$\Deltaglobal<\infty$, or we reach the prescribed budget of
oracle calls. The rule \eqref{eq:new_infeas_rule}, which is
to be executed before Step~\ref{step:algo_my_cov_method_for}
of Algorithm~\ref{algo:my_cov_method}, is motivated by
\eqref{prob:cov_method} in Algorithm~\ref{algo:cov_method}.
In contrast to \eqref{prob:relax_surrogate_prob} which does
not emphasize on the value of $H$,
\eqref{eq:new_infeas_rule} seeks to minimize $H$ in an
effort to find a feasible point for \eqref{prob:orig_prob}.
Note that this modification only affects the constants of
$\Tsuff$ prescribed in
Theorem~\ref{thm:my_cov_method}\ref{thm:my_cov_method_tsuff},
thanks to Proposition~\ref{prop:BoundOnTAlgoOne}.

\subsubsection{Tightness of the worst-case analysis}

Lemma~\ref{lem:yudin} recalls a known lower bound on 
the oracle call complexity for constrained, global
optimization~\cite[Sec. 1.6]{yudin}. 
\begin{lem}[\textsc{Worst-case necessary
    budget}]\label{lem:yudin} For any $\nu>0$ and $K>0$, there exist
    twice-differentiable $J\in\FK$ and $H\in\FK$ such that
    any first-order sequential optimization algorithm 
    takes $\left\lceil
\frac{C}{\nu^{(d/2)}}\right\rceil$
    oracle calls to compute a $\left(0,\nu \frac{K \diam(
    \mathcal{X})^2}{8},\nu\frac{K \diam(
    \mathcal{X})^2}{8}\right)$-minimum. Here, $C$ is a positive
constant that is independent $\nu$.  
\end{lem}

Using Lemma~\ref{lem:yudin}, we conclude that for any
$\eta>0$, there is a problem
instance of \eqref{prob:orig_prob} with twice-differentiable
objective and constraint functions for which
Algorithm~\ref{algo:my_cov_method} needs at
least $\left\lceil \frac{C}{{\left(2\sqrt{2}\right)}^d}\diam(
    \mathcal{X})^d\left(\frac{K
}{\eta}\right)^\frac{d}{2}\right\rceil$ oracle calls to
compute a $\left(0,\eta,\eta\right)$-minimum of
\eqref{prob:orig_prob}.  On the other hand, even for such an
``adversarial'' problem instance,
Algorithm~\ref{algo:my_cov_method} needs at most $
\left\lceil{2{\left(\sqrt{d}\right)}^d{\diam(
\mathcal{X})}^d{\left(\frac{K}{\eta}\right)}^\frac{d}{2}
}\right\rceil + 1$ oracle calls to compute a
$\left(\eta,\eta\right)$-minimum of
\eqref{prob:orig_prob} by
Theorem~\ref{thm:my_cov_method}\ref{thm:my_cov_method_tsuff}, 
provided the user-specified relaxation threshold
$\delta\leq \frac{K\eta^2}{2\GJmax^2}$ and $(H,\delta)$ satisfies Assumption~\ref{assum:good_grad_my_cov_method}.
In other words, $\Tsuff$ prescribed for
Algorithm~\ref{algo:my_cov_method} is sufficient and
necessary (up to constant factors) for a large subclass of
problems of the form \eqref{prob:orig_prob}.

\subsection{Global optimization of \eqref{prob:orig_prob}
    for smooth, strongly-convex $H$
without constraint violation}
\label{sub:CMSE}

Algorithm~\ref{algo:my_cov_method_mu_convex} addresses
Problem~\ref{p_st:algo_safe} to compute a
near-global  minimum for \eqref{prob:orig_prob} without any
constraint violation. It requires the constraint function
$H$ be strongly-convex with a
known convexity constant $\mu>0$ and a feasible initial
solution guess $q_1$. Unlike
Algorithm~\ref{algo:my_cov_method},
Algorithm~\ref{algo:my_cov_method_mu_convex} does not
require Assumption~\ref{assum:good_grad_my_cov_method} or
the $\delta$-relaxation \eqref{prob:orig_prob_relax_delta}.

Algorithm~\ref{algo:my_cov_method_mu_convex} follows a
relax-and-project approach to create a
monotonically-decreasing sequence of outer-approximations
$\{\Htsc \leq 0\}$ and a monotonically-increasing sequence of
inner-approximations $\{H_t^+ \leq 0\}$ of the \emph{a priori}
unknown set $H\leq 0$. Thanks to the $\mu$-convexity of
$H$ and a feasible initial solution guess $q_1$, these 
approximations are non-empty.
Algorithm~\ref{algo:my_cov_method_mu_convex} ensures that 
the queries $q_{t+1}$ are feasible for 
\eqref{prob:orig_prob} for every $t\in\Nint{T-1}$ via a
projection step \eqref{prob:cov_method_scvx_project}. 
Consequently,
Algorithm~\ref{algo:my_cov_method_mu_convex} computes a
$(0,\Deltaglobal,0)$-minimum of \eqref{prob:orig_prob} at
every iteration, and accommodates non-convex, smooth
objective functions, similar to
Algorithm~\ref{algo:my_cov_method}.

\begin{thm}[\textsc{Algorithm~\ref{algo:my_cov_method_mu_convex}
    addresses Problem~\ref{p_st:algo_safe}}]\label{thm:my_cov_method_mu_cvx}
    For any objective function $J$,
    $\mu$-convex constraint function $H$, and feasible initial solution
    $q_1$, the following statements about
    Algorithm~\ref{algo:my_cov_method_mu_convex}:
    {\renewcommand{\theenumi}{\Alph{enumi}}
    \begin{enumerate}
        \item \textsc{(Well-definedness)} The optimization
            problems \eqref{prob:cov_method_scvx_outer} and
            \eqref{prob:cov_method_scvx_project} in
            Algorithm~\ref{algo:my_cov_method_mu_convex} are
            always feasible and have a finite optimal
            solution at all iterations.\label{thm:my_cov_method_mu_cvx_feas_all}
        \item \textsc{(No constraint violation)} All queries
            of Algorithm~\ref{algo:my_cov_method_mu_convex}
            are feasible for \eqref{prob:orig_prob}. 
            \label{thm:my_cov_method_mu_cvx_feas_queries}
        \item \textsc{(Anytime property)} Let
            Algorithm~\ref{algo:my_cov_method} run up to an
            iteration $t\in \mathbb{N}$. Then,
            $\xddagglobal$ is a $(0,\Deltaglobal,0)$-minimum of
            \eqref{prob:orig_prob}.\label{thm:my_cov_method_mu_cvx_anytime}
        \item \textsc{(Worst-case, sufficient budget)}
            For a suboptimality threshold $\eta > 0$, 
            define $\kappa\triangleq L_J\left(\frac{L_H\GJmax}{2L_J\GHmax} +
        \frac{2L_H
        }{\mu}\right)$, and 
    \begin{align}
        \Tsuffsc \triangleq \left\lceil{{\left(\diam(
                \mathcal{X})
\sqrt{d}\right)}^d{\left(\frac{\kappa}{\eta}\right)}^\frac{d}{2}}\right\rceil
+ 1.\end{align}
    Then, Algorithm~\ref{algo:my_cov_method_mu_convex}
    terminates with an $(0,\eta,0)$-minimum of
    \eqref{prob:orig_prob} at some iteration $t\leq
    \Tsuffsc$.  \label{thm:my_cov_method_mu_cvx_tsuff}
    \end{enumerate}
    }
\end{thm}
\begin{proof}
    \emph{Proof of~\ref{thm:my_cov_method_mu_cvx_feas_all})}
    The optimization problems
    \eqref{prob:cov_method_scvx_outer} and
    \eqref{prob:cov_method_scvx_project} always admit $q_1$
    as a feasible solution $H(q_1)=\Htsc(q_1)=H_t^+(q_1)$,
    and therefore have a non-empty feasible solution space.
    Furthermore, since $ \mathcal{X}$ is compact, these
    optimization problems have a well-defined global minima.

    \emph{Proof
    of~\ref{thm:my_cov_method_mu_cvx_feas_queries})} 
    The proof of feasibility of $q_{t+1}$ for
    \eqref{prob:orig_prob} at every iteration $t \in
    \Nint{T-1}$ follows from the observation that
    $H(q_{t+1})\leq H_t^+(q_{t+1}) \leq 0$ by
    \eqref{prob:cov_method_scvx_project} and
    Lemma~\ref{lem:approx}. 

    \emph{Proof
    of~\ref{thm:my_cov_method_mu_cvx_anytime})} Since
    $\Htsc\leq H$, \eqref{prob:cov_method_scvx_outer} is a
    relaxation of \eqref{prob:orig_prob} with the constraint
    $H\leq 0$ relaxed to $\Htsc\leq 0$.  Consequently, we
    have the following inequality at every iteration
    $t\in\Nint{T-1}$ (similar to
    \eqref{eq:mycovmethod_ineq}),
    \begin{align}
        J_t^-(\xi_{t+1}) \leq J_t^-(\xopt) \leq J(\xopt) \leq J(\xddagglobal)
        \leq \min_{i\in\Nint{t+1}}J(q_i) \leq J(q_{t+1}).\label{eq:global_estim_bounds}
    \end{align}
    From \eqref{eq:global_estim_bounds}, 
    we have the following bounds on the true global
    suboptimality,
    \begin{align}
        0\leq J(\xddagglobal) - J(\xopt)&\leq \Deltaglobal
        \triangleq \min_{i\in\Nint{t+1}}J(q_i) -
        J_t^-(\xi_{t+1}).\label{eq:mycovmethod_scvx_global_ineq}
    \end{align}
    Thus, $\xddagglobal$ is a $(0,\Deltaglobal,0)$-minimum
    of \eqref{prob:orig_prob} at every iteration since
    $H(\xddagglobal)\leq 0$ by
    Theorem~\ref{thm:my_cov_method_mu_cvx}\ref{thm:my_cov_method_mu_cvx_feas_queries}
    and \eqref{eq:cov_method_scvx_min}.

    \emph{Proof
    of~\ref{thm:my_cov_method_mu_cvx_tsuff})}
    We seek a lower bound on the budget $T$ of oracle calls,
    which ensures $\Deltaglobal \leq \eta$. 
    Using mean value theorem, Lemma~\ref{lem:approx}, the definition of 
    $\GJmax$, and
    \eqref{eq:global_estim_bounds}, we characterize the following upper
    bound on $\Deltaglobal$,
    \begin{align}
        \Deltaglobal&\leq J(q_{t+1}) -
        J_t^-(\xi_{t+1}) \nonumber \\
                    &= J(q_{t+1}) - \max\limits_{i\in
            \Nint{t}}\left(\ell(\xi_{t+1};q_i,J)-\frac{L_J}{2} \|
            \xi_{t+1} -
   q_i\|^2\right) \nonumber \\
                    &= J(q_{t+1}) - J(\xi_{t+1}) +
                    J(\xi_{t+1}) - \max\limits_{i\in
                    \Nint{t}}\left(\ell(\xi_{t+1};q_i,J)-\frac{L_J}{2}
                \| \xi_{t+1} - q_i\|^2\right) \nonumber \\
                    &\leq \GJmax\|q_{t+1} -
                    \xi_{t+1}\| + L_J
                    \min_{i\in\Nint{t}}\|\xi_{t+1} - q_i\|^2.\label{eq:simple_deltaglobal_ub}
    \end{align}
    We will upper bound
    \eqref{eq:simple_deltaglobal_ub} using
    $\min_{i\in\Nint{t}} \|q_{t+1} - q_i\|^2$ to complete
    the proof using Lemma~\ref{lem:pidgeonhole}.
    
    For the given budget $T$, define $\Pi_T\subseteq
    \Nint{T-1}$ as the finite set of iterations where
    $q_{t+1} =  \xi_{t+1}$, i.e.,
    \eqref{prob:cov_method_scvx_project} resulted in a
    trivial projection. At any iteration $t$, we have two
    cases --- $t\in\Pi$ or $t\not\in\Pi$. For the first
    case, we have $q_{t+1} =  \xi_{t+1}$, which implies
    \begin{align}
        \Deltaglobal&\leq L_J
        \min_{i\in\Nint{t}} \|q_{t+1} -
        q_i\|^2,\label{eq:upper_bound_scenario_1}.
    \end{align}
    
    We now consider the second case --- $t\not\in\Pi$, where
    \eqref{prob:cov_method_scvx_project} generates a
    non-trivial projection point $q_{t+1}\neq \xi_{t+1}$.
    Here, we will upper bound the terms in
    \eqref{eq:simple_deltaglobal_ub} separately to
    characterize the sufficient budget. We will show that
    the proposed upper bound for the second case subsumes
    \eqref{eq:upper_bound_scenario_1} to complete the proof
    using Lemma~\ref{lem:pidgeonhole} and
    \eqref{eq:mycovmethod_scvx_global_ineq}.
    
    First, we show the following upper bound to the second term in
    \eqref{eq:simple_deltaglobal_ub},
    \begin{align}
        \min_{i\in\Nint{t}}\|\xi_{t+1} -
        q_i\|^2 &\leq \frac{2\GHmax}{\mu}\|q_{t+1} -
        \xi_{t+1}\| + \frac{L_H}{\mu}\min_{i\in\Nint{t}}\|q_{t+1} - q_i\|^2.\label{eq:H_ineq}
    \end{align}
    To prove \eqref{eq:H_ineq}, we first recall that
    $\Htsc(\xi_{t+1})\leq 0 = H_t^+(q_{t+1})$  from
    \eqref{prob:cov_method_scvx_outer} and
    \eqref{prob:cov_method_scvx_project} at every
    $t\not\in\Pi$. Consequently, 
    \begin{alignat}{3}
                 & & \Htsc(\xi_{t+1})\triangleq &\max_{i\in\Nint{t}}
                \Big(\ell(\xi_{t+1};q_i, H) + \frac{\mu}{2}\|\xi_{t+1} -
                    q_i\|^2\Big) \nonumber \\
                                 & & &\leq H_t^+(q_{t+1}) \triangleq 
                    \min_{i\in\Nint{t}}
                    \left(\ell(q_{t+1};q_i, H)
                        + \frac{L_H}{2}\|q_{t+1} -
q_i\|^2\right) \nonumber \\
                \Rightarrow\ & \forall i\in\Nint{t},\ &\mu\|\xi_{t+1} -
        q_i\|^2 &\leq 2\nabla H(q_i) \cdot (q_{t+1} -
        \xi_{t+1}) 
                + L_H\|q_{t+1} - q_i\|^2, \nonumber \\
                \Rightarrow\ & \forall i\in\Nint{t},\ &\mu\|\xi_{t+1} -
        q_i\|^2 &\leq 2\|\nabla H(q_i)\|\|q_{t+1} -
        \xi_{t+1}\| + L_H\|q_{t+1} - q_i\|^2, \nonumber \\
                \Leftrightarrow\ &\forall i\in\Nint{t},\ &\|\xi_{t+1} -
                q_i\|^2 &\leq \frac{2\GHmax}{\mu}\|q_{t+1} -
                \xi_{t+1}\| + \frac{L_H}{\mu}\|q_{t+1} -
                q_i\|^2 \Rightarrow \eqref{eq:H_ineq}.
        \nonumber
    \end{alignat}%
    Here, we used Lemma~\ref{lem:approx}, $\mu$-convexity of
    $H$ \eqref{eq:approx_sc}, Cauchy-Schwartz inequality,
    and the definition of $\GHmax$. 
    Substituting \eqref{eq:H_ineq} in
    \eqref{eq:simple_deltaglobal_ub}, we have
    \begin{align}
        \Deltaglobal&\leq \left({\GJmax +
        \frac{2L_J\GHmax}{\mu}}\right)\|q_{t+1} -
        \xi_{t+1}\| + \frac{L_J
        L_H}{\mu}\min_{i\in\Nint{t}}\|q_{t+1} -
        q_i\|^2.\label{eq:simple_deltaglobal_ub_2}
    \end{align}

    Next, we characterize an upper bound for the projection
    distance $\|q_{t+1} - \xi_{t+1}\|$ at every such
    iteration $t\not\in\Pi$. From
    \eqref{prob:cov_method_scvx_project}, we have $\|
    \xi_{t+1} - q_{t+1}\| \leq \| \xi_{t+1} - q_{i}\|$ for
    every $i\in\Nint{t}$ and $t\not\in\Pi$. Consequently,
    $\| \xi_{t+1} - q_{t+1}\| \leq \sqrt{\min_{i\in\Nint{t}} \| \xi_{t+1} -
    q_{i}\|^2}$. Using \eqref{eq:H_ineq},
    \begin{align}
        \|q_{t+1} - \xi_{t+1}\| \leq
        \sqrt{\frac{2\GHmax}{\mu}\|q_{t+1} - \xi_{t+1}\| +
        \frac{L_H}{\mu}\min_{i\in\Nint{t}}\|q_{t+1} -
    q_i\|^2}.\label{eq:min_xi_q_ub}
    \end{align}
    Recall that the maximum value of a positive $z$ that
    satisfies the inequality $z\leq \sqrt{2az + b}$ for any
    $a,b>0$ occurs at $z=\sqrt{a^2 + b} - a$. For
    a fixed $a>0$, the function $f_a(b)=\sqrt{a^2 + b} - a$ is
    concave in $b$. Consequently, $z \leq f_a(b)\leq
    \ell(b;0,f_a)=\frac{b}{2a}$, where $\ell(b;0,f_a)$ is
    the first-order approximation of $f_a$ about $b=0$ for
    some fixed $a$. For $a=\frac{\GHmax}{\mu}$ and
    $b=\frac{L_H}{\mu}\min_{i\in\Nint{t}}\|q_{t+1} -
    q_i\|^2$ in \eqref{eq:min_xi_q_ub}, we have
    \begin{align}
        \|q_{t+1} - \xi_{t+1}\| &\leq \sqrt{a^2 + b} - a 
        \leq \frac{b}{2a}\leq \frac{L_H}{2\GHmax}\min_{i\in\Nint{t}}\|q_{t+1} -
    q_i\|^2.\label{eq:projection_ub}
    \end{align}
    Finally, substituting \eqref{eq:projection_ub} into
    \eqref{eq:simple_deltaglobal_ub_2}, we obtain
    \begin{align}
        \Deltaglobal&\leq \kappa\min_{i\in\Nint{t}}\|q_{t+1} -
        q_i\|^2,\text{ with }\kappa\triangleq {L_J\left(\frac{L_H\GJmax}{2L_J\GHmax} +
        \frac{2L_H
}{\mu}\right)}.\label{eq:simple_deltaglobal_ub_3}
    \end{align}
    The upper bound \eqref{eq:simple_deltaglobal_ub_3} also
    upper bounds \eqref{eq:upper_bound_scenario_1}, since
    $\frac{\kappa}{L_J} \geq \frac{2L_H
    }{\mu}\geq 2$.
    Therefore, we can guarantee
    $\Deltaglobal\leq \eta$, when $T$ is chosen such that
    for some $t\in\Nint{T-1}$, 
    $\min_{i\in\Nint{t}}\|q_{t+1} - q_i\|^2 \leq
    \frac{\eta}{\kappa}$. We complete the proof using
    Lemma~\ref{lem:pidgeonhole} and
    \eqref{eq:mycovmethod_scvx_global_ineq}.
\end{proof}

\subsubsection{Tightness of the worst-case analysis}

Proposition~\ref{prop:my_cov_method_mu_convex_is_tight}
shows that Algorithm~\ref{algo:my_cov_method_mu_convex} is
\emph{worst-case} optimal in the user-specified
suboptimality threshold $\eta$. We only focus on first-order
sequential optimization algorithms that can solve
\eqref{prob:orig_prob} with $\mu$-convex $H$ and guarantee
no constraint violation.

\begin{prop}\label{prop:my_cov_method_mu_convex_is_tight}
    For every first-order sequential optimization algorithm
    that solves \eqref{prob:orig_prob} without producing any
    queries in the infeasible set $H>0$, $K_J\geq 0$, and
    $L_H\geq 0$
    there is a problem instance of \eqref{prob:orig_prob}
    with twice-differentiable, smooth $J \in \FKJ$ and $\mu$-convex,
    smooth $H\in \FHmu$ such that the algorithm requires at least
    $\left\lceil{
{\left(\frac{K_J}{\eta}\right)}^\frac{d}{2}}C'\right\rceil$  queries
    to compute an $(0,\eta,0)$-minimum. Here, $C'$ is a
    positive constant independent of true Lipschitz gradient constant
    of the objective function $K_J$ and $\eta$.
\end{prop}
\begin{proof}
    We first note that, under the given assumptions, 
    \eqref{prob:orig_prob} is equivalent to the following
    optimization problem,
    \begin{align}
        \mathrm{minimize}\ J(x)\quad\mathrm{subject\
        to}\ x \in
        \mathcal{Y}\triangleq(\{H\leq 0\}\cap
        \mathcal{X})\label{prob:orig_prob_Y}.
    \end{align}
    By definition, the set $\mathcal{Y}$ is closed (since
    $H$ is continuous), convex (since $H$ is convex), and
    bounded (since $H$ is $\mu$-convex with $\mu>0$) and $
    \mathcal{X}$ is convex and compact. Thus,
    \eqref{prob:orig_prob_Y} is similar to
    \eqref{prob:orig_prob_noH} with $ \mathcal{X}$
    restricted to \emph{a priori unknown}, convex, and
    compact set $ \mathcal{Y}$.
    Recall
    that for every first-order sequential optimization
    algorithm designed to solve \eqref{prob:orig_prob_noH}
    when $ \mathcal{Y}$ is \emph{known}, there is a
    twice-differentiable, smooth objective function $J$ for
    which the algorithm takes $\left\lceil{
    {\left(\frac{K_J}{\eta}\right)}^\frac{d}{2}}C'\right\rceil$
    queries to compute an $(0,\eta,0)$-minimum for some
    positive constant $C'>0$~\cite[Thm.
    4]{vavasis1995complexity}. Clearly, the \emph{necessary}
    bound must also hold for the case where the set $
    \mathcal{Y}$ is \emph{a priori unknown}, and the algorithms
    query only within the set $ \mathcal{Y}$. This 
    completes the proof.  
\end{proof}

From
Theorem~\ref{thm:my_cov_method_mu_cvx}\ref{thm:my_cov_method_mu_cvx_tsuff}
and Proposition~\ref{prop:my_cov_method_mu_convex_is_tight},
the sufficient budget $\Tsuffsc$ is necessary and sufficient
(up to a constant factor) to address
Problem~\ref{p_st:algo_safe}, irrespective of the choice of
the objective and the constraint function or the feasible
initial solution.

\subsubsection{Is strong-convexity of $H$ necessary in
Problem~\ref{p_st:algo_safe}?}

Assume, for contradiction, that there is some first-order
sequential optimization algorithm $\mathscr{A}$ that solves
\eqref{prob:orig_prob} without any constraint violation and
requiring $H$ to be only smooth and not necessarily
strongly-convex.
By Whitney's theorem, for
every constraint function $H$ and the associated
sequence of \emph{feasible} iterates
${\{q_t\}}_{t\in\Nint{T}}$ generated by $\mathscr{A}$, there
exists a constraint function $H'\in\FH$ such that the
first-order oracles of $H$ and $H'$ agree at all iterations
$i\in\Nint{t-1}$ for some $t\in\Nint{T}$, but $H'(q_t) > 0
\geq H(q_t)$. In other words, there always exist a problem
instance for which $\mathscr{A}$ will violate the
constraint at iteration $t$, specifically the problem
instance \eqref{prob:orig_prob} with $H'$ as the constraint
function instead of $H$.

To show that Algorithm~\ref{algo:my_cov_method_mu_convex}
can fail to solve \eqref{prob:orig_prob} when $H$ is just
convex, but not strongly-convex, consider the constraint
function as the \emph{zero function} $H\triangleq 0$. In
this case, every iterate of
Algorithm~\ref{algo:my_cov_method_mu_convex} is $q_1$, the
user-provided initial solution guess, since the set
$\{H_t^+(x)\leq 0\}=\{q_1\}$ at every iteration $t\in\Npos$.

\subsection{Tractable implementation of
    Algorithms~\ref{algo:my_cov_method}
and~\ref{algo:my_cov_method_mu_convex}}

Algorithms~\ref{algo:my_cov_method}
and~\ref{algo:my_cov_method_mu_convex} require global
optimization of \emph{non-convex} optimization problems
\eqref{prob:relax_surrogate_prob},
\eqref{prob:cov_method_scvx_outer}, and
\eqref{prob:cov_method_scvx_project}. We now discuss
tractable approaches to solve these optimization problems.

\newcommand{\cmoeOriginal}{
        \arraycolsep=3pt
        \def\arraystretch{1.5}
        \begin{equation}
            \small
            \hskip-0.5em\boxed{\begin{array}{rl}
                \underset{x}{\mathrm{min.}}&\max\limits_{i\in\Nint{t}}
            \left({\ell(x;q_i,J) - \frac{L_J}{2}} \|x -
            q_i\|^2\right)\\
                \mathrm{s.\ t.}&x\in \mathcal{X},\\
                               &\max\limits_{j\in\Nint{t}}
                               \left({\ell(x;q_j,H) -
            \frac{L_H}{2} \|x - q_j\|^2}\right)\leq 0
        \end{array}} \nonumber
        \end{equation}
}
\newcommand{\cmpeOriginal}{
        \arraycolsep=3pt
        \def\arraystretch{1.5}
        \begin{equation}
            \boxed{
            \small
            \hskip-0.5em\begin{array}{rl}
            \underset{x}{\mathrm{min.}}&\max\limits_{i\in\Nint{t}}
            \left({\ell(x;q_i,J) - \frac{L_J}{2}} \|x -
            q_i\|^2\right)\\
                \mathrm{s.\ t.}&x\in \mathcal{X},\\
                               &\max\limits_{j\in\Nint{t}}
                               \left({\ell(x;q_j,H) +
            \frac{\mu}{2} \|x - q_j\|^2}\right)\leq 0
        \end{array}} \nonumber
        \end{equation}
}
\newcommand{\cmoeEpigraphProblemArray}{
    \begin{array}{rl}
        \underset{x,u}{\mathrm{minimize}}& L_J\left(u -
            \frac{x\cdot x}{2}\right) \\
            \mathrm{subject\ to}& u\in \mathbb{R},\ x\in \mathcal{X}\\
            \forall i \in \Nint{t},&
            \frac{\ell(x;q_i,J)}{L_J} +
            \ell(x;q_i,Q) \leq u\\
            \forall j \in \Nint{t},&
            \frac{\ell(x;q_j,H)}{L_H} +
            \ell(x;q_j,Q) \leq \frac{x\cdot x}{2}\\
    \end{array}
}
\newcommand{\cmoeEpigraph}{
        \arraycolsep=3pt
        \def\arraystretch{1.5}
        \begin{equation}
            \small
            \boxed{\cmoeEpigraphProblemArray} \nonumber %
        \end{equation}
}
\newcommand{\cmpeEpigraphProblemArray}{
    \begin{array}{rl}
        \underset{x,u}{\mathrm{minimize}}& L_J\left(u -
            \frac{x\cdot x}{2}\right) \\
        \mathrm{subject\ to}& u\in \mathbb{R},\  x\in \mathcal{X}\\
        \forall i \in \Nint{t},&
        \frac{\ell(x;q_i,J)}{L_J} +
        \ell(x;q_i,Q) \leq u\\
        \forall j \in \Nint{t}, & \frac{\ell(x;q_j,H)}{\mu}
        - \ell(x;q_j,Q) + \frac{x \cdot x}{2}\leq 0\\
    \end{array}
}
\newcommand{\cmpeEpigraph}{
        \arraycolsep=3pt
        \def\arraystretch{1.5}
        \begin{equation}
            \small
            \boxed{\cmpeEpigraphProblemArray} \nonumber %
        \end{equation}%
}

\begin{figure}[!h]
\adjustbox{width=1\linewidth}{
\begin{tabular}{p{0.5\linewidth}   p{0.5\linewidth}}
    {\hfill \textbf{\eqref{prob:relax_surrogate_prob} for
    Algorithm~\ref{algo:my_cov_method}}\hfill} & {\hfil
\textbf{\eqref{prob:cov_method_scvx_outer} for
Algorithm~\ref{algo:my_cov_method_mu_convex}}\hfill}\\[2ex]
    \multicolumn{2}{c}{$\Updownarrow\hspace*{1em}$
    {\small
    \emph{Expand $J_t^-, H_t^-$ and $\Htsc$ using
Lemma~\ref{lem:approx} and
\eqref{eq:approx_sc}}}
$\hspace*{1em}\Updownarrow$}\\
\cmoeOriginal & \cmpeOriginal \\
\multicolumn{2}{c}{$\Updownarrow$\hspace*{1em}
    \small
        \begin{tabular}{c}
        \emph{Define $Q(x)\triangleq\frac{x\cdot x}{2}$,
            which implies $Q(x) -
        \ell(x;q_i,Q) = \frac{\|x - q_i\|^2}{2}$;}\\
        \emph{Piecewise-linear minimization~\cite[Sec.
            4.3.1]{BoydConvex2004} via epigraph formulation}
       \end{tabular}
    \hspace*{1em}$\Updownarrow$}\\
    \cmoeEpigraph & \cmpeEpigraph \\
\end{tabular}}
\caption{Reformulation of \eqref{prob:relax_surrogate_prob}
    and \eqref{prob:cov_method_scvx_outer}. Both of the
    optimization problems minimizes a concave quadratic
    function subject to convex constraints from $
    \mathcal{X}$, $t$ linear constraints, and $t$
    quadratic constraints. While the quadratic constraints
    in \eqref{prob:relax_surrogate_prob} are non-convex, the
quadratic constraints in \eqref{prob:cov_method_scvx_outer}
are convex.}\label{fig:reformulation}
\end{figure}

Figure~\ref{fig:reformulation} sketches a reformulation of
\eqref{prob:relax_surrogate_prob} and
\eqref{prob:cov_method_scvx_outer}. The resulting problems
are non-convex, quadratically-constrained quadratic programs, when the constraint $x\in\mathcal{X}$ can
be expressed as a collection of linear/second-order cone
constraints. We utilize \texttt{GUROBI}, a commercial off-the-shelf
solver, to solve such 
problems. \texttt{GUROBI} can tackle
\eqref{prob:relax_surrogate_prob} and
\eqref{prob:cov_method_scvx_outer} via spatial branching~\cite{GUROBI}. The
reformulation of \eqref{prob:relax_surrogate_prob} also
shows that Algorithm~\ref{algo:my_cov_method} simplifies to
a \emph{minimax space-filling} design-based
optimization~\cite{pronzato_minimax_2017} for very large
$L_J$ and $L_H$.

The optimization problem
\eqref{prob:cov_method_scvx_project} seeks the projection of
a point $\xi_{t+1}\in \mathcal{X}$ onto the set $\{H_t^+
\leq 0\}$ at every iteration $t\in \Nint{T}$. From
\eqref{eq:major_f} and simple algebraic manipulations, we
see that 
\begin{align}
    \{H_t^+ \leq 0\}= \bigcup_{i\in\Nint{t}}
    \mathrm{Ball}\left(q_i - \frac{\nabla H(q_i)}{L_H},
    \sqrt{\frac{\|\nabla H(q_i)\|^2}{L_H^2} - \frac{2
H(q_i)}{L_H}} \right)\label{eq:union_of_balls}
\end{align}
Consequently, we can solve
\eqref{prob:cov_method_scvx_project} \emph{exactly} in two
steps: 1) compute the projection point of $\xi_{t+1}$ onto
the $t$ balls separately (available in closed-form), and 2)
choose among the $t$ projection points, the point closest to
$\xi_{t+1}$ via a finite minimum operation.
Recall that for any $z \in \mathcal{X}$, 
the projected point is $z$ if
$z\in\mathrm{Ball}(c,r)$, otherwise the projected point is $c +
\frac{r}{\|z - c\|}(z - c)$.

\section{Numerical experiments}
\label{sec:num}

We used Python to perform all computations on an Intel
i7-4600U CPU with 4 cores, 2.1GHz clock rate and 7.5 GB RAM. 

\subsection{Benchmarking against existing approaches:
Solution quality and scalability} 
\label{sub:benchmark}

\newcommand{\functionwidth}{0.63\linewidth}
\begin{table}[h]
    \small
    \centering
    \begin{tabular}{ccccc}
    \toprule
    Problem & $J$ & $H$ & Infeasible $q_1$ & Feasible $q_1$ \\
    \midrule
    P1 & \texttt{Br}   & \multirow{2}{*}{\texttt{SinQ}}  &
    \multirow{2}{*}{$\left(-{\sqrt{\frac{40
    \pi}{3}},\sqrt{-\frac{40
    \pi}{3}}}\right)$} & \multirow{2}{*}{$\left({\sqrt{\frac{20
    \pi}{3}},\sqrt{\frac{20
    \pi}{3}}}\right)$} \\
    P2 & \texttt{MBr}  & & & \\
    P3 & \texttt{Br}   & \multirow{2}{*}{\texttt{MBr}} &
    \multirow{2}{*}{(5.5, -9)} & \multirow{2}{*}{(0, 10)}\\
    P4 & \texttt{MBr}  & & &\\
    P5 & \texttt{Br}   & \multirow{2}{*}{\texttt{InvBowl}} &
    \multirow{2}{*}{$c_\mathrm{bowl} +
    \frac{R_\mathrm{bowl}}{2}\left(\frac{1}{\sqrt{2}},\frac{1}{\sqrt{2}}\right)$} & \multirow{2}{*}{(5.5,
-9)}\\
    P6 & \texttt{MBr}  & & &\\
    P7 & \texttt{Br}   & \multirow{2}{*}{\texttt{Bowl}} &
    \multirow{2}{*}{(5.5, -9)} & \multirow{2}{*}{$c_\mathrm{bowl} +
    \frac{R_\mathrm{bowl}}{2}\left(\frac{1}{\sqrt{2}},\frac{1}{\sqrt{2}}\right)$}\\
    P8 & \texttt{MBr}  & & &\\
    \bottomrule
    \end{tabular} 
    \ \\\vspace*{1em}
    \begin{tabular}{p{0.135\linewidth}ccc}
    \toprule
    Name  & Function & $K_f$ & $L_f$ \\[1ex]
    \midrule
    Branin (\texttt{Br}) & 
    \begin{minipage}[t]{\functionwidth}
        \ \\[-2.5ex]
        $\texttt{Br}(x_1,x_2) = (x_2 - \frac{5.1}{4\pi^2} x_1^2 + \frac{5}{\pi} x_1 -
        6)^2 + 10(1-\frac{1}{8\pi}) \cos(x_1) + 10$.       
    \end{minipage} & $\approx65$ & $75$ \\[1ex] \hline
    Modified Branin~(\texttt{MBr}) &  
    \begin{minipage}[t]{\functionwidth}
        $\texttt{MBr}(x_1,x_2) =
    \texttt{Br}(x_1,x_2) + 20 x_1 - 30 x_2$.       
\end{minipage} & $\approx65$ & $75$ \\ \hline
    Bowl (\texttt{Bowl}) &  
    \begin{minipage}[t]{\functionwidth}
        \ \\[-2.5ex]
        $\texttt{Bowl}(x_1,x_2) = \frac{1}{2}\left(\|x - c_\mathrm{bowl}\|^2 - R_\mathrm{bowl}^2\right)$ where
        $R_\mathrm{bowl}=10$ and $c_\mathrm{bowl}=[-3,-3]$.       
    \end{minipage} & $1$ & $2$ \\[3.5ex] \hline
    Inverted bowl (\texttt{InvBowl}) &  
    \begin{minipage}[t]{\functionwidth}
        $\texttt{InvBowl}(x_1,x_2) = - \texttt{Bowl}(x_1,
        x_2)$.       
\end{minipage} & $1$ & $2$ \\ \hline
    Sine-quadratic (\texttt{SinQ}) &  
    \begin{minipage}[t]{\functionwidth}
        \ \\[-1.5ex]
        $\texttt{SinQ}(x_1,x_2) = \sin\left(\frac{x_1^2 +
        x_2 ^2}{10}\right)$.       
\end{minipage} & $\approx 4.2$ & $6$ \\
    \bottomrule
\end{tabular}
   \caption{Benchmark problems defined using
       two-dimensional functions with $
       \mathcal{X}=[-10,10]^2$.  The true Lipschitz constant
       for the gradients $K_f$, when unknown, are computed
       via gridding.  See~\cite[Sec.
       B]{kochenderfer2019algorithms} for the definition of
   the Branin function. We use strong-convexity constant
   $\mu=0.5$ for \texttt{Bowl} function, which is smaller
   than its true strong-convexity constant of one.}
   \label{tab:prob_desc}
\end{table}

We consider several benchmark problems to compare the
performance of Algorithms~\ref{algo:my_cov_method}
and~\ref{algo:my_cov_method_mu_convex} with existing
approaches to solve \eqref{prob:orig_prob} --- bayesian
optimization and local optimization.  For Bayesian
optimization, we considered the constrained expected
improvement (\texttt{cEI})~\cite{gardner2014bayesian}
approach as implemented in
\texttt{emukit}~\cite{emukit2019}. \texttt{emukit} solves
the resulting unconstrained, non-convex, acquisition
optimization problem approximately using Limited-memory
Broyden-Fletcher-Goldfarb-Shanno algorithm via random
starts. We also considered \texttt{SLSQP}, a first-order
local optimization algorithm as implemented in Python's
\texttt{scipy} package~\cite{2020SciPy-NMeth}.  

We investigate the quality
of the computed solutions in terms of their global
suboptimality, the compute time, and the number of
infeasible queries. We also discuss the near-infeasibility
certificates computed by Algorithm~\ref{algo:my_cov_method},
and its ability to deal with moderately-dimensioned
problems.

\subsubsection{Solution quality}

Table~\ref{tab:prob_desc} lists the eight benchmark
problems in the form of \eqref{prob:orig_prob}. Here, we chose 
the set $ \mathcal{X}=[-10,10]^2$, the thresholds $\eta=0.1$
and $\delta=10^{-5}$, and a budget of $T=400$. 

For the Bayesian optimization, we
considered $5$ independent trials to account for the
stochastic behavior of the \texttt{emukit}'s implementation
of \texttt{cEI}.  We used grid search followed by a
``polishing step'' using local optimization \texttt{SLSQP}
to approximate the true global minimum of
\eqref{prob:orig_prob} $J(\xopt)$ by $J(x^\ast_g)$. For grid
search, we used a step size of $0.05$, which resulted in
$40,000$ oracle queries (excluding the queries in the
polishing step). 

\begin{figure}
    \centering
    \newcommand{\trimValuesTrue}{0 0 0 8}
    \newcommand{\trimValuesFalse}{0 25 0 0}
    \includegraphics[width=1\linewidth,Trim=\trimValuesFalse,clip]{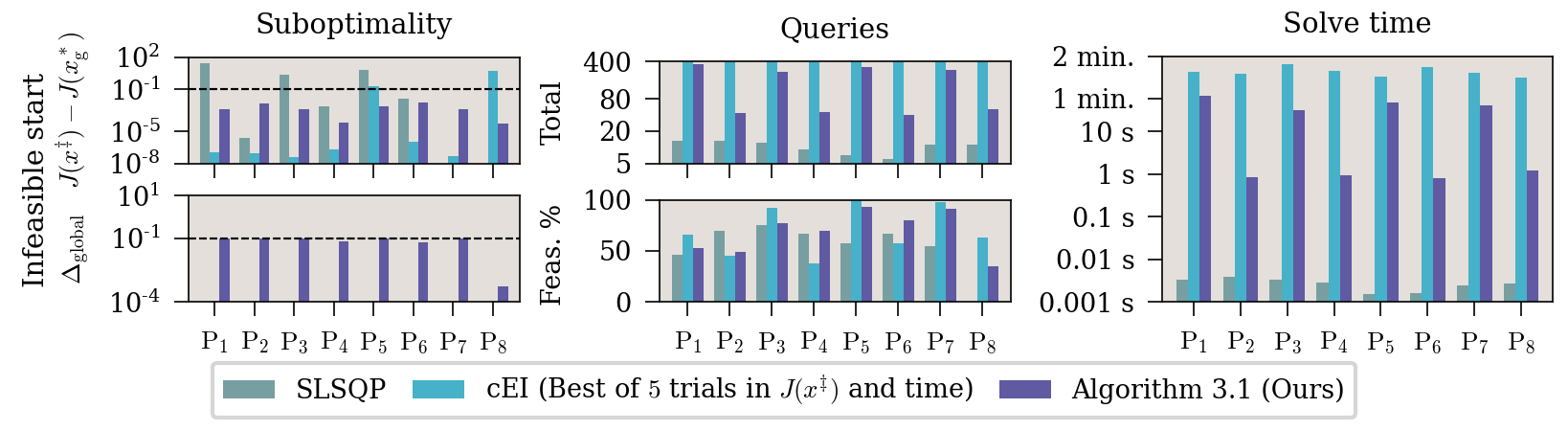} 
    \includegraphics[width=1\linewidth,Trim=\trimValuesTrue,clip]{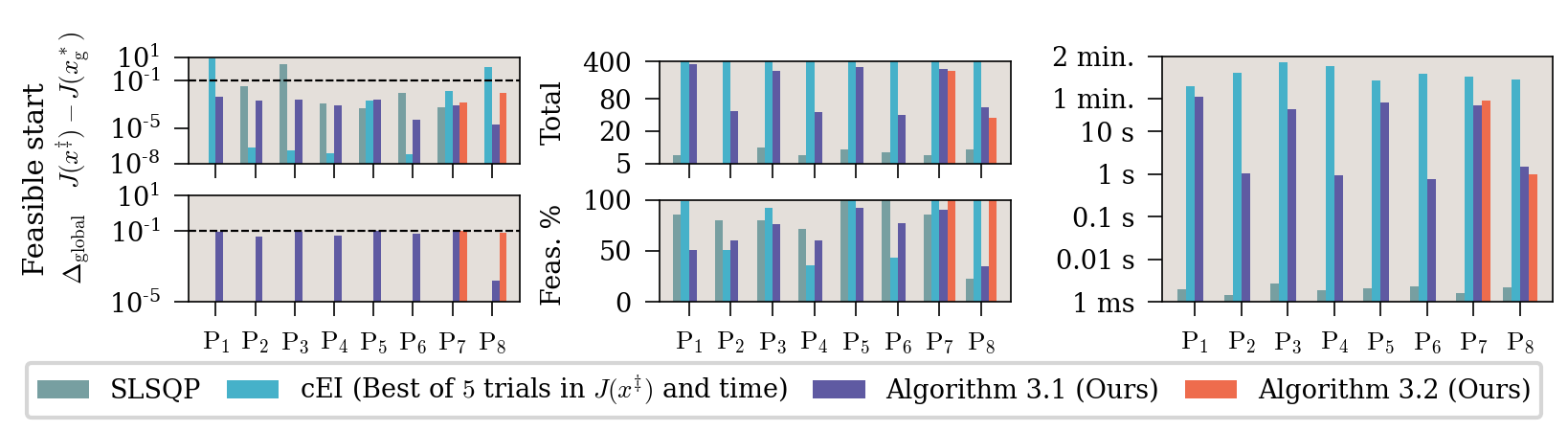} 
    \caption{
        Results of the benchmark problems
        (Table~\ref{tab:prob_desc}) for a budget $T=400$.
        Algorithm~\ref{algo:my_cov_method} computes a
        $(0.1,10^{-5})$-minimum for every problem,  
        irrespective of the feasibility of the initial
        solution guess. For Problems $P_7$ and $P_8$ with
        \emph{a priori} unknown, strongly-convex constraint $H\leq 0$, 
        Algorithm~\ref{algo:my_cov_method_mu_convex}
        computes a $(0.1,0)$-minimum without any constraint
        violation. Both of the algorithms return an upper
        bound $\Deltaglobal$ on the true global
        suboptimality of the computed solution, $J(x^\ddag)
        - J(\xopt_g)$. They outperform Bayesian optimization
        (\texttt{cEI}) in computation time, and recover a
    global minimum unlike \texttt{SLSQP}, a local
optimization method.}\label{fig:performance}
\end{figure}
\begin{figure}
    \centering
    \newcommand{\trimValuesRuntimeTrue}{0 0 0 0}
    \newcommand{\trimValuesRuntimeFalse}{0 31 0 0}
    \includegraphics[width=1\linewidth,Trim=\trimValuesRuntimeFalse,clip]{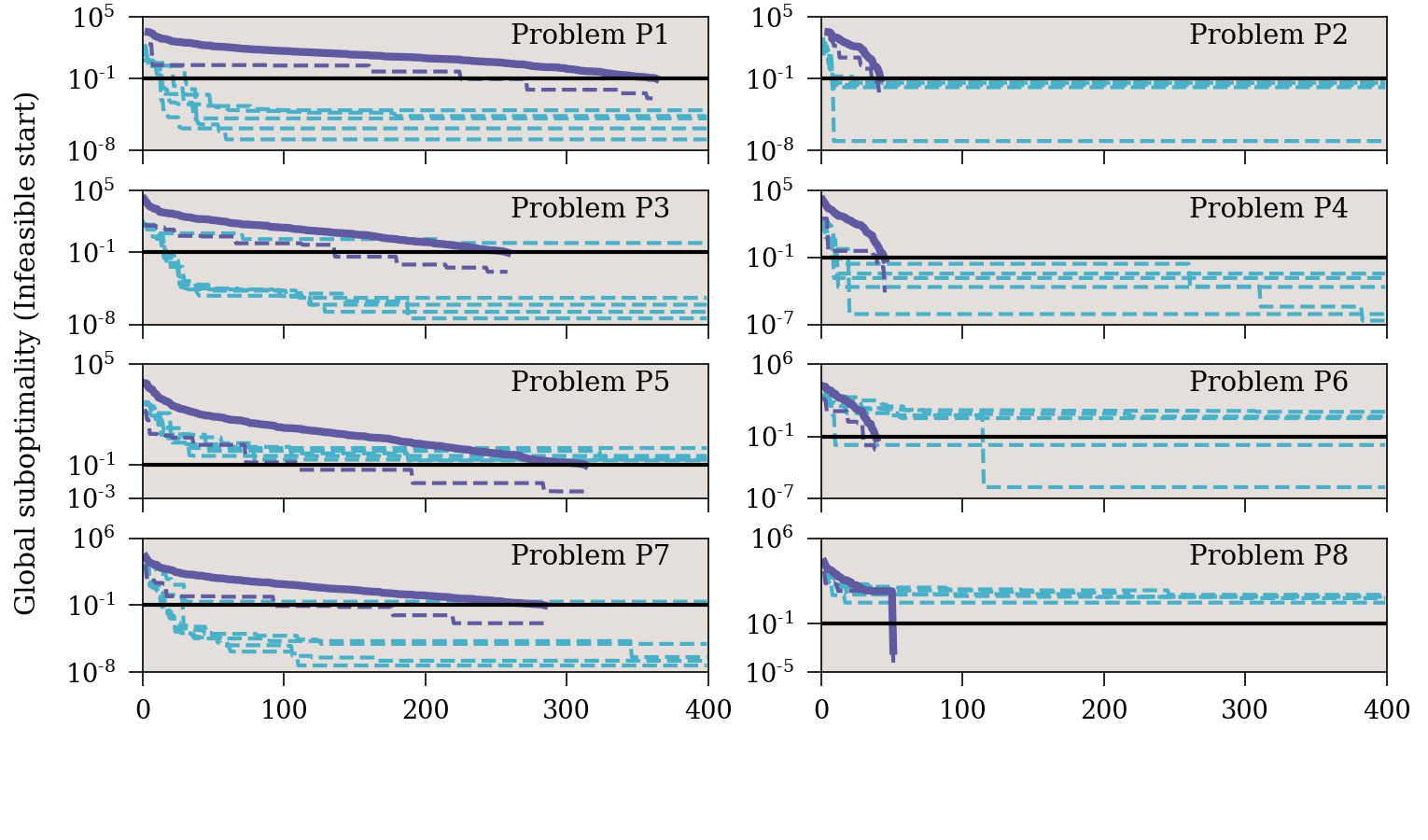}
    \par\noindent\rule{\textwidth}{0.4pt}
    \ \vspace*{-0.5em}\\
    \includegraphics[width=1\linewidth,Trim=\trimValuesRuntimeTrue,clip]{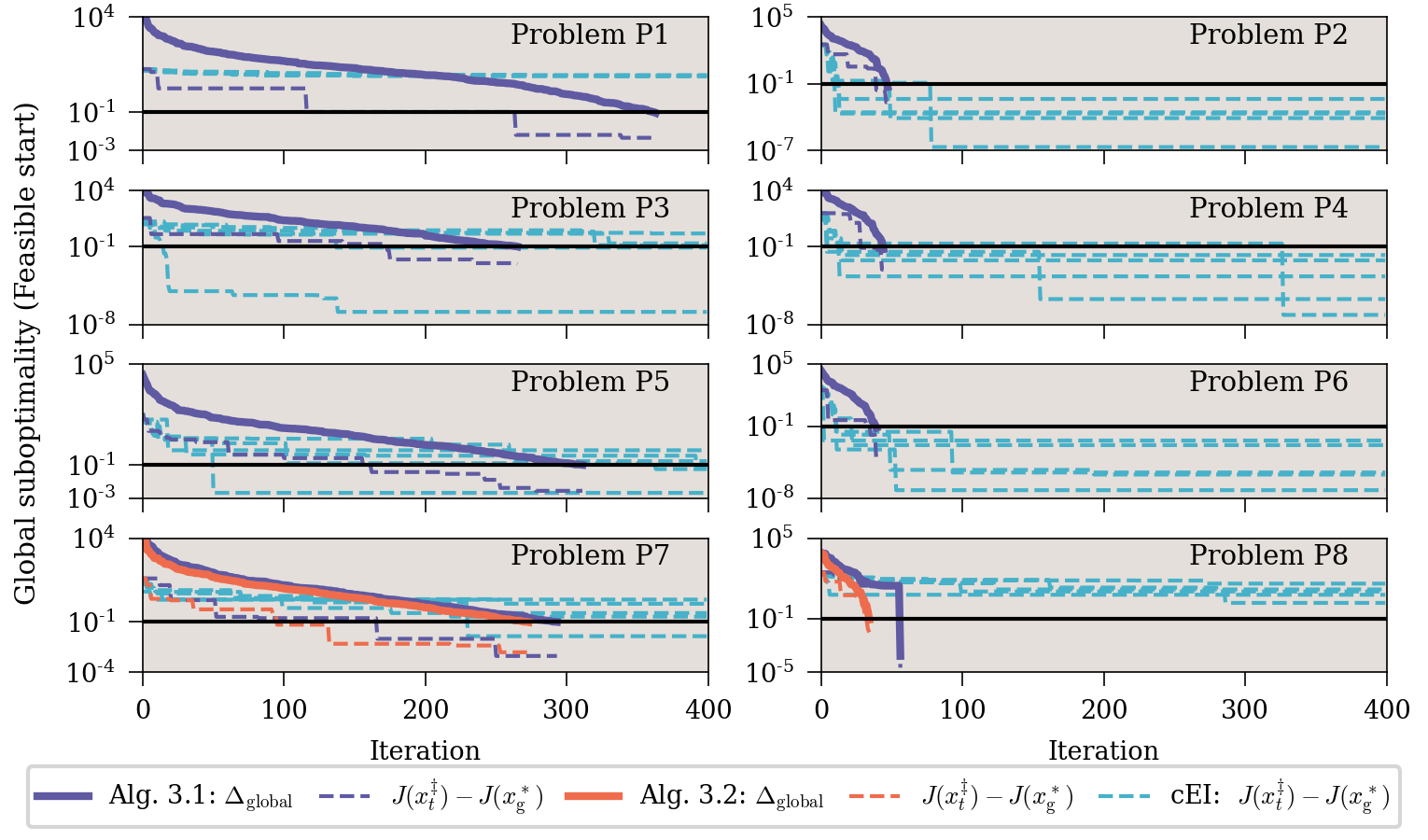}
    \caption{Global suboptimality of iterates returned by Algorithms~\ref{algo:my_cov_method}
        and~\ref{algo:my_cov_method_mu_convex}, and Bayesian
        optimization (\texttt{cEI}). Algorithms~\ref{algo:my_cov_method}
        and~\ref{algo:my_cov_method_mu_convex} provide a
        global suboptimality upper bound $\Deltaglobal$ at
        every iteration. In Problems $P_2$, $P_4$, $P_6$,
        and $P_8$, $\Deltaglobal$ enables early termination, resulting in
significant gains in computation time.}\label{fig:performance_runtime}
\end{figure}
\begin{figure}
    \centering
    \newcommand{\trimValuesScalability}{0 0 0 2}
    \includegraphics[width=1\linewidth,Trim=\trimValuesScalability,clip]{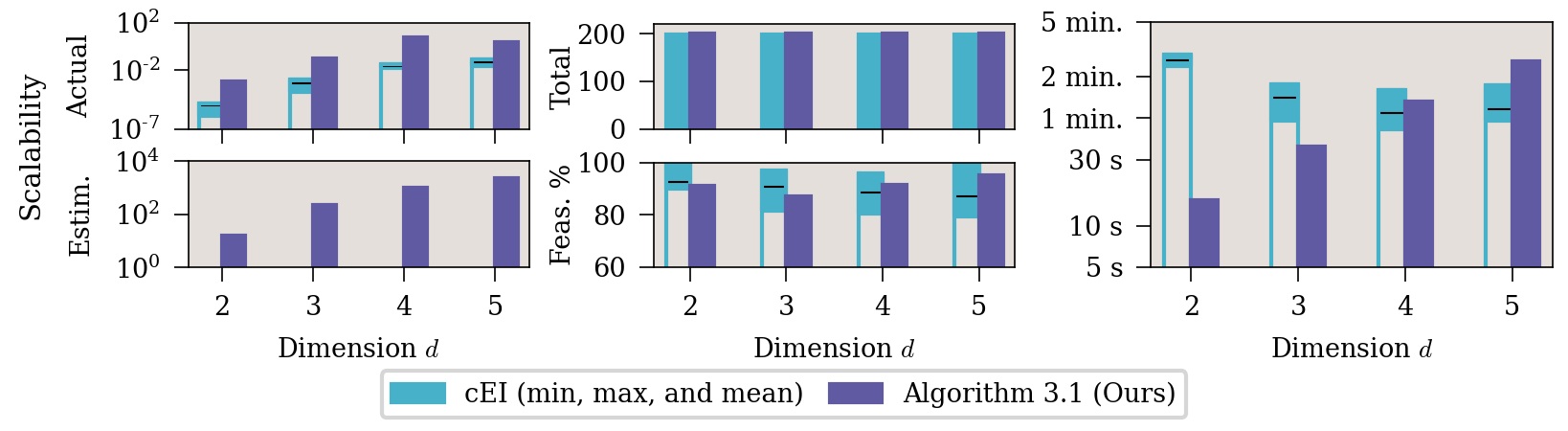}
    \caption{Algorithm~\ref{algo:my_cov_method} and
        \texttt{cEI} applied to moderately-high
    dimensional problems.}\label{fig:scalability}
\end{figure}

Figure~\ref{fig:performance} summarizes the results. We see 
that Algorithm~\ref{algo:my_cov_method} computes the
($\delta$-relaxed) global optima within the budget for every
problem in Table~\ref{tab:prob_desc}, while
Algorithm~\ref{algo:my_cov_method_mu_convex} computes the
global optima without any constraint violation whenever the
unknown constraint function $H$ is strongly-convex (Problems
$P_7$ and $P_8$). Both of the proposed algorithms are
significantly faster 
than Bayesian optimization (\texttt{cEI}) in most of the
trials. In addition,
Bayesian optimization currently lacks the guarantee
constraint-violation-free optimization when $H$ is
strongly-convex, and produces infeasible queries in contrast
to Algorithm~\ref{algo:my_cov_method_mu_convex}. As
expected, the local optimization method \texttt{SLSQP} can
return suboptimal solutions (Problems $P_1,P_3$, and $P_5$),
but converges significantly faster than
Algorithms~\ref{algo:my_cov_method}
and~\ref{algo:my_cov_method_mu_convex} and Bayesian
optimization. Unlike the proposed algorithms, \texttt{SLSQP}
can converge to an infeasible solution
(Problems $P_1,P_7$, and $P_8$).

Figure~\ref{fig:performance_runtime} compares the global
suboptimality of the intermediate iterates returned by
Algorithms~\ref{algo:my_cov_method}
and~\ref{algo:my_cov_method_mu_convex} and Bayesian
optimization, with respect to $J(\xopt_g)$. As expected, the global
suboptimality bound $\Deltaglobal$ computed by Algorithms~\ref{algo:my_cov_method}
and~\ref{algo:my_cov_method_mu_convex} at each iteration upper bounds the true
global suboptimality $J(x^\ddag_t) - J(\xopt_g)$, and demonstrates the anytime property of these
algorithms. Empirically, we see that the upper bound
$\Deltaglobal$ is not severely conservative, and tracks the
decrease in the true global suboptimality well. The upper
bound $\Deltaglobal$ helps
Algorithms~\ref{algo:my_cov_method}
and~\ref{algo:my_cov_method_mu_convex} terminate early in
Problems $P_2$, $P_4$, $P_6$, and $P_8$, while guaranteeing
the satisfaction of the desired global suboptimality
threshold of $\eta=0.1$. 
In most of the problems, the trials of the Bayesian
optimization approach (\texttt{cEI}) and the iterates of
Algorithms~\ref{algo:my_cov_method}
and~\ref{algo:my_cov_method_mu_convex} do not dominate each
other.

\subsubsection{(Near-)Infeasibility certificates}

Algorithm~\ref{algo:my_cov_method} can produce
infeasibility or near-infeasibility certificates for
infeasible instances of \eqref{prob:orig_prob}. To
illustrate the utility of such a certificate, consider
the following infeasible optimization problem in $x$,
\begin{align}
    \mathrm{minimize}\ J(x)= \mathrm{MBr}(x)\quad\mathrm{subject\
    to}\ H(x)=\mathrm{Br}(x)\leq 0\label{prob:infeas}.
\end{align}
The infeasibility proof of \eqref{prob:infeas} follows from
the fact that the global minimum value of
$\mathrm{Br}(x)$ is strictly positive~\cite[Sec.
B.3]{kochenderfer2019algorithms}, which implies that $\{
\mathrm{Br}(x)\leq 0\}=\emptyset$. We chose
$\mathcal{X}=[-10,10]^2$ and a budget of $T=400$. Using the
modification of Algorithm~\ref{algo:my_cov_method} given in
\eqref{eq:new_infeas_rule}, we found $\gamma=0.32$ after
exhausting the budget in $58.05$ seconds ($<1$ minute). In
other words, Algorithm~\ref{algo:my_cov_method} proves that 
\begin{align}
    \mathrm{minimize}\ J(x)= \mathrm{MBr}(x)\quad\mathrm{subject\
    to}\ H(x)=\mathrm{Br}(x)\leq -0.32,
\end{align}
is infeasible (Definition~\ref{defn:gamma_infeas}).  Without
the modification, the near-infeasibility certificate
returned by Algorithm~\ref{algo:my_cov_method} is much
higher $\gamma=256.98$ after exhausting the budget in
$97.33$ seconds ($<2$ minutes).

\subsubsection{Scalability}

For scalability evaluation for $d\in\{2,3,4,5\}$, we
considered the optimization problem \eqref{prob:orig_prob}
with $J$ as the $d$-dimensional Rosenbrock's function and
constraint $H=\texttt{InvBowl}$. Recall that the global
minimum $\xopt$ of $J$ over $ \mathcal{X}=[-10,10]^d$ is
$1_d$, a $d$-dimensional vector of
ones~\cite[Sec. B.6]{kochenderfer2019algorithms}. We set
constraint $H=\texttt{InvBowl}$ with $c_\mathrm{Bowl}$ as
the mid point of line joining the global minimum and the
initial point $q_1=[-10,-10,-10,\ldots]\in \mathbb{R}^d$
(one of the vertices of $ \mathcal{X}$). We chose
$R_\mathrm{Bowl}=0.4\|\xopt - q_1\|$ to ensure that the
selected initial point $q_1$ remains feasible. We chose
$L_J=60$ and $L_H=2$.

Figure~\ref{fig:scalability} shows that
Algorithm~\ref{algo:my_cov_method} can be applied to
\eqref{prob:orig_prob} with moderate values of $d$ as well.
While the actual suboptimality of the optimization problem
remains low, we found that the suboptimality bounds become
loose as $d$ increases, potentially due to
$\eta^\frac{-d}{2}$ dependence on the sufficient budget
(Theorem~\ref{thm:my_cov_method}\ref{thm:my_cov_method_mu_cvx_tsuff}).
In addition, the computational time of
Algorithm~\ref{algo:my_cov_method}
increases with $d$, potentially due to the reliance on
mixed-integer optimization to solve
\eqref{prob:relax_surrogate_prob}.

\subsection{Training a neural network with constraints} 
\label{sub:mountain}

Next, we apply Algorithm~\ref{algo:my_cov_method} to an
instance of \eqref{prob:orig_prob} arising from policy
optimization in machine learning. 
Specifically, we consider the policy optimization for the
classical mountain car problem~\cite{openaigym}, where we
train a policy neural network $\texttt{NN}(\theta)$ with
five network parameters.  We seek a policy that drives the
car to reach the top of the mountain within a predetermined
number of steps. 

\begin{wrapfigure}{r}{.4\textwidth}
    \definecolor{pydarkgreen}{RGB}{0, 100, 0}
    \definecolor{pydarkorange}{RGB}{255, 140, 0}
    \centering
    \vspace*{-3em}
    \includegraphics[width=0.99\linewidth]{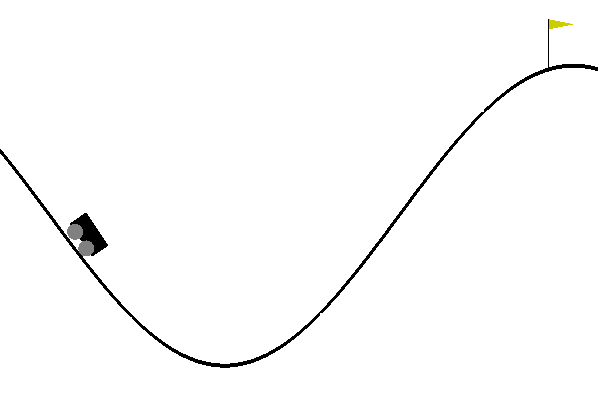}
    \begin{tikzpicture}[overlay,remember picture, >=stealth]
        \draw[thick, dashed, <->] (2.15, 0.3) --  (0, 0.3)
            node[midway, below, align=right] {\small $0.45$};
        \draw[thin] (2.15, 0.2) --  (2.15, 0.6);
        \draw[thick, dashed, <->] (2.15, 0.62) -- (2.15, 3.2)
        node[anchor=east,
        midway, above, rotate=90, align=right] {\small Goal
        height};

        \draw[thick] (-2.5, 0.62) -- (2.5, 0.62);
    
        \node[align=right, text width=13em, yshift=11.5em]
            (0, 0) {\small $\texttt{Height}(x)=0.45 \sin(3x)
            + 0.55$};

        \draw[thick, dashed, ->] (0, 3.8) -- (-1.65, 1.9);

        \draw[thin] (0, 0.2) --  (0, 0.6);
        \node[yshift=2.2em] (0, 0.6) {\small $0$};

        \draw[thin] (-1.8, 0.2) --  (-1.8, 2.5);
        \draw[thick, dashed, <->] (-1.8, 0.3) --  (0, 0.3)
            node[midway, below, align=right] {\scriptsize
            $\texttt{car\_position}$};
    \end{tikzpicture}
    \caption{Mountain car example.}
    \vspace*{-4em}
    \label{fig:mountaincar}
\end{wrapfigure}
 
\emph{Policy neural network:}
We consider a network with
two input nodes and one output nodes using $\tanh$
activation function. The network is defined as follows,
\begin{subequations}
\begin{align}
    y_1&=\tanh(\theta_1 \times \texttt{car\_position})\\
    y_2&=\tanh(\theta_2 \times \texttt{car\_velocity})\\
    u&=\tanh(\theta_3 y_1 + \theta_4 y_2 + \theta_5)
\end{align}\label{eq:neural_network}%
\end{subequations}%
Here, $\theta_i$ are the network parameters or the decision
variables. We fixed the domain of network parameters
$\Theta=\{z:-\theta_{\max} \leq z \leq
\theta_{\max}\}\subset \mathbb{R}^5$ with $\theta_{\max}=5
\times [\frac{1}{1.2}, \frac{1}{0.07},1,1,1]$.  Here, we
have normalized states (position and velocity of the
car) by their bounds.

\emph{Physics-driven constraints:} In general, policy
optimization is a hard
problem~\cite{sutton2018reinforcement}, since the mapping
from the policy parameters to the reward is highly
non-convex, and the large number of policy parameters
prevents tractable enumeration. A natural way to reduce the
search space is to enforce additional constraints on the
problem. 
Consider the following policy
optimization problem,
\begin{align}
    \begin{array}{cl}
        \underset{ \theta \in
        \Theta\subset
        \mathbb{R}^5}{\mathrm{max.}}&J=\overbrace{\texttt{CumulativeRewardOverAnEpisode}(\texttt{NN}(\theta))}^{\text{Unknown
    function of $\theta$}} \\
        \mathrm{s.\ t.}&
        H=\underbrace{\texttt{TotalEnergyAtEndOfEpisode}
    (\texttt{NN}(\theta))}_{\text{(Unknown function of
$\theta$)}} \geq
    \underbrace{\texttt{GoalPotentialEnergy}}_{\text{Known
constant}}.
    \end{array} \nonumber %
\end{align}
The motivation for imposing constraints on the total energy 
arise from the observation that successful policies that
drive the car to the top of the mountain should also inject
sufficient energy into the car. Here, we compute the total
energy of the system and the goal potential energy at the
end of the episode as follows with $g=9.8$,
\begin{align}
    &\texttt{TotalEnergyAtEndOfEpisode}
    (\texttt{NN}(\theta)) \nonumber \\
    &\qquad= \texttt{Potential energy} + \texttt{Kinetic
    energy} \nonumber \\
    &\qquad= g\times \texttt{Height}(\texttt{terminal\_car\_position}) +
    \frac{\texttt{terminal\_car\_velocity}^2}{2}\nonumber \\
    &\texttt{GoalPotentialEnergy} = g\times
    \texttt{Height}(\texttt{goal\_position\_x}) \nonumber %
\end{align}
We can safely ignore the mass of the car since it appears on
both sides of the constraint. We declare that the task is
completed successfully, when the cumulative reward is above
$90$~\cite{openaigym}.

Note that $J$ and $H$ are smooth functions of the policy
parameters $\theta$, since we have used \texttt{tanh} as the
activation function in \eqref{eq:neural_network}. We compute
the gradients $\nabla J$ and $\nabla H$ via finite
differences (step size of $0.01$) and choose sufficiently
large $L_J=L_H=100$. We chose a budget of $T=10$, which
translates to $110$ episodes for finite difference-based
gradient computation.

We found that Algorithm~\ref{algo:my_cov_method} computed a
policy neural network completes the task successfully. On
the other hand, when the energy constraints were not
imposed, we did not meet the minimum reward threshold for
success, possibly due to the low number of episodes. 

\section{Conclusion}
\label{sec:conc}

This paper introduces two novel algorithms for constrained
global optimization of \emph{a priori unknown} functions
with Lipschitz continuous gradients. The proposed approaches
are inspired by the existing literature in covering method
to global optimization problems. They accommodate
finite budget of oracle calls and terminate with non-trival
global suboptimality guarantees. The first approach
accommodates infeasible start and returns near-global
minimum or a (near-)infeasibility certificate. 
The second approach guarantees
feasible iterates when the unknown constraint function is
strongly-convex and the initial solution guess is feasible.
We also characterize the necessary and sufficient budget of
oracle calls required to  satisfy user-specified tolerances
for a large class of optimization problems. Empirical
studies show the efficacy of these approaches.

\appendix

\section{Adversarial instance of $H$ in
Section~\ref{sub:seqoptalg}} 
\label{app:adversarial_example}
We construct a \emph{resistive} oracle for the constraint
functions to meet the requirements specified in
Section~\ref{sub:seqoptalg}, and it sufficies to consider a
single-constraint case $M=1$. Resistive oracles for
sequential optimization algorithms do not commit to a
specific $H$, but adapt based on the queries. Analyzing the
algorithm's performance under such oracles reveals its
worst-case performance. See~\cite{nesterov_lectures_2018}
for more details.

\textsc{Desirable properties of the resistive oracle for
$H$}: Given $T\in \Npos$,  $L_H > 0$, and $T$ oracle queries
arising from \emph{any} sequential optimization algorithm,
we can construct a first-order oracle for some $H\in \FH$
such that: 
\begin{enumerate}
    \item all of the $T$ oracle calls 
        returns $H > 0$ and $\nabla H=0$, i.e., all of the
        query points requested by a sequential optimization
        algorithm are infeasible for \eqref{prob:orig_prob},
        and
    \item there exists $y\in \mathcal{X}$ distinct from the
        $T$ query points such that $y$ is feasible for
        \eqref{prob:orig_prob}.
\end{enumerate}
In other words, given $T$ and $L_H$, the constructed oracle
responds to the queries of any sequential optimization
algorithm such that the algorithm can ``discover" the
feasibility of the constrained optimization problem
\eqref{prob:orig_prob}, only at the ${(T+1)}^\mathrm{th}$
query. Since every algorithm is bound by the budget of the
oracle calls, it is forced to declare infeasibility based on
the infeasible $T$ queries. 

\textsc{Construction of the resistive oracle for $H\in\FH$}:
Let $\mathcal{Q}_T \triangleq \{q_i: i\in \Nint{T}\}$ be the
set of query points corresponding to the first $T$ oracle
calls from the sequential algorithm under study. We define
$y\in\arg\sup_{x \in \mathcal{X}} \min_{i\in\Nint{T}} \| x -
q\| \in \mathcal{X}$, a point in $\mathcal{X}$ that is the
furthest away from $\mathcal{Q}_T$. By Whitney's extension
theorem~\cite{whitney1934analytic}, there is always a
function $h: \mathcal{X} \to \mathbb{R}$ with Lipschitz
continuous gradient, such that $h(q_i)>0$ and $\nabla
h(q_i)=0$ for $i\in\Nint{T}$, and $h(y)\leq 0$. While the
constructed $h$ need not lie in $\FH$ as desired, we can
always construct the desired $H\in\FH$ via $H=\alpha h$ for
some appropriate scaling $\alpha > 0$. This completes the
construction.

\section{Illustrative example on page \pageref{tab:illustration}}
\label{app:example}
For the first example, we study the following non-convex
optimization problem,
\begin{align}
    \begin{array}{rl}
    \mathrm{minimize}&\quad J(x)=\frac{\sin(x)}{2x} - 0.02x\\
    \mathrm{subject\ to}&\quad x\in
    \mathcal{X}=[-10,10],\quad H(x)=\frac{(x-6)^2(x+6)^2 -
    900}{4000}\leq 0\\
    \end{array}\label{prob:safeopt_wontwork}%
\end{align}
with Lipschitz gradient constants as $L_J\in\{0.2,1\}$ and
$L_H=0.2$. We choose suboptimality threshold $\eta=0.01$,
and relaxation threshold $\delta=10^{-8}$.

For the second example, we study the following non-convex
optimization problem with strongly-convex constraint
function $H$,
\begin{align}
    \begin{array}{rl}
    \mathrm{minimize}&\quad J(x)=\frac{\sin(x)}{2x} - 0.02x\\
    \mathrm{subject\ to}&\quad x\in
\mathcal{X}=[-10,10],\quad H(x)=\frac{(x-1)^2 - 7^2}{100}\leq 0\\
    \end{array}\label{prob:safeopt_wontwork_2}%
\end{align}
with Lipschitz gradient constants as
$L_J=0.2$ and $L_H=1.2$, and convexity constant
$\mu=0.01$. We choose suboptimality threshold $\eta=0.01$.

\end{document}